\documentclass{amsart}

\usepackage[latin1]{inputenc}
\usepackage[T1]{fontenc}
\usepackage{enumerate}
\usepackage{mathtools}
\usepackage{amssymb}
\usepackage{mathrsfs}
\usepackage{mathabx}

\newtheorem{theorem}{Theorem}[section]
\newtheorem{lemma}[theorem]{Lemma}
\newtheorem{cor}[theorem]{Corollary}
\newtheorem{prop}[theorem]{Proposition}

\theoremstyle{definition}
\newtheorem{definition}[theorem]{Definition}
\newtheorem{example}[theorem]{Example}
\newtheorem{problem}{Problem}

\theoremstyle{remark}
\newtheorem{remark}[theorem]{Remark}

\numberwithin{equation}{section}

\newcommand{\next}{\text{next}}
\newcommand{\K}{\mathbb K}

\newcommand{\C}{\mathbb C}

\title[Invariant subspaces for non-normable Fréchet spaces]{Invariant subspaces for non-normable Fréchet spaces}
\author[Q. Menet]{Quentin Menet}

\address{Quentin Menet, Univ. Artois, EA 2462, Laboratoire de Mathématiques de Lens (LML), F-62300 Lens, France}
\email{quentin.menet@univ-artois.fr}

\allowdisplaybreaks[3]
\begin{document}
\begin{abstract}
A Fréchet space $X$ satisfies the Hereditary Invariant Subspace (resp. Subset) Property if for every closed infinite-dimensional subspace $M$ in $X$, each continuous operator on $M$ possesses a non-trivial invariant subspace (resp. subset). In this paper, we exhibit a family of non-normable separable infinite-dimensional Fréchet spaces satisfying the Hereditary Invariant Subspace Property and we show that many non-normable Fréchet spaces do not satisfy this property. We also state sufficient conditions for the existence of a continuous operator without non-trivial invariant subset and deduce among other examples that there exists a continuous operator without non-trivial invariant subset on the space of entire functions $H(\mathbb{C})$.
\end{abstract}
\maketitle
\section{Introduction}

A topological vector space $X$ is a Fréchet space if the topology of $X$ is induced by a non-decreasing and separating sequence of seminorms $(p_j)_{j\ge 1}$ and if the metric given by $d(x,y)=\sum_{j\ge 1}2^{-j}\min\{1,p_j(x-y)\}$ is complete. We will then write $(X,(p_j)_j)$ in order to refer to the Fréchet space $X$ endowed with the topology induced by the seminorms $(p_j)$ and we will always assume that the sequence $(p_j)$ is chosen non-decreasing and separating. We remark that if $(\lambda_n)_{n\ge 1}$ is a non-decreasing sequence of positive numbers and $(j_n)_{n\ge 1}$ is an increasing sequence then the sequence of seminorms $(p_j)_j$ induces the same topology than the sequence of seminorms $(\lambda_n p_{j_n})_n$. In particular, we can choose to work with a subsequence of the sequence $(p_j)_j$ when it is necessary. 

We denote by $L(X)$ the set of continuous operators on $X$. We recall that if $(X,(p_j)_j)$ is a Fréchet space then a linear map $T$ on $X$ is continuous if and only if for every $j\ge 1$, there exists $C>0$ and $J\ge 1$ such that for every $x\in X$, $p_j(Tx)\le C p_J(x)$. In particular, the seminorms $p_j$ and $p_J$ in these inequalities can be nonequivalent. This fact will play a key role in our constructions of continuous operators satisfying desired properties.

The goal of this paper consists in trying to better understand the Invariant Subset Property and the Invariant Subspace Property for non-normable Fréchet spaces, i.e. Fréchet spaces which are not Banach spaces.

\begin{definition}
Let $(X,(p_j)_j)$ be a Fréchet space and $T\in L(X)$. 
\begin{itemize}
\item We say that $T$ possesses a non-trivial invariant subspace (resp. subset) if there exists a closed subspace (resp. subset) $M$ different from $\{0\}$ and $X$ such that $TM\subset M$.
\item We say that $X$ satisfies the Invariant Subspace Property (resp. Invariant Subset Property) if every operator $T\in L(X)$ possesses a non-trivial invariant subspace (resp. subset).
\item We say that $X$ satisfies the Hereditary Invariant Subspace Property (resp. Hereditary Invariant Subset Property) if every closed infinite-dimensional subspace of $X$ satisfies the Invariant Subspace Property (resp. Invariant Subset Property).
\end{itemize}
\end{definition}

If $X$ is a complex finite-dimensional Fréchet space with $\dim X\ge 2$ then $X$ satisfies the Invariant Subspace Property since every continuous operator on $X$ possesses an eigenvector. On the other hand, if $X$ is a non-separable Fréchet space then $X$ also satisfies the Invariant Subspace Property since for every continuous operator $T$ on $X$, it suffices to consider the invariant subspace $\overline{\text{span}}\{T^nx:n\ge 0\}$ where $x$ is  some non-zero vector. We thus restrict our study to separable infinite-dimensional Fréchet spaces.

In the case of Banach spaces, we know thanks to Enflo and Read that there exist separable infinite-dimensional Banach spaces which do not satisfy the Invariant Subspace Property~\cite{Enflo1, Enflo2, Read1} and even separable infinite-dimensional Banach spaces which do not satisfy the Invariant Subset Property~\cite{Read3}. A classical example is given by the Banach space $\ell_1$~\cite{Read2,Read3}. This means that there exists a continuous operator $T$ on $\ell_1$ such that for each non-zero vector $x\in  \ell_1$, the orbit $\text{Orb}(x,T):=\{T^nx:n\ge 0\}$ is dense in $\ell_1$. In 2011, Argyros and Haydon~\cite{Argyros1} gave an example of separable infinite-dimensional Banach space satisfying the Invariant Subspace Property. The existence of non-trivial invariant subspaces for any continuous operator on this Banach space comes from the fact that every continuous operator on this spaces has the form $I+\lambda K$ where $K$ is a compact operator~\cite{Bernstein, Halmos, Lomonosov}. Finally, in 2014, Argyros and Motakis~\cite{Argyros2} modified the space constructed by Argyros and Haydon in order to get the first example of separable infinite-dimensional Banach space satisfying the Hereditary Invariant Subspace Property. This space is also the first example of reflexive Banach spaces satisfying the Invariant Subspace Property and it is an important open question to know if every reflexive separable infinite-dimensional Banach space satisfies the Invariant Subspace Property.

In the case of non-normable Fréchet spaces, an example of separable infinite-dimensional non-normable Fréchet space satisfying the Invariant Subspace Property is known for a long time. Indeed, It has been shown in 1969 by Körber~\cite{Korber} and later by Shields~\cite{Shields} that the space $\omega$ satisfies this property, where $\omega$ is the space of all complex sequences endowed with the seminorms $p_j(x)=\max_{n\le j}|x_n|$. It was even shown by Johnson and Shields~\cite{Johnson} that every operator on $\omega$ (that is not a scalar mulitple of the identity) possesses a non-trivial hyperinvariant subspace. Since every closed infinite-dimensional subspace is isomorphic to $\omega$ (see Theorem 2.6.4 and  Corollary 2.6.5 in \cite{Bonet}), we deduce in particular that $\omega$ satisfies the Hereditary Invariant Subspace Property. The first examples of separable infinite-dimensional non-normable Fréchet space which do not satisfy the Invariant Subspace Property were given by Atzmon~\cite{Atzmon, Atzmon2}. In 2012, Goli\'{n}ski~\cite{Golinski1} exhibited  classical examples by showing that the space of entire functions $H(\mathbb{C})$ and the Schwartz space of rapidly decreasing functions $s$ do not satisfy the Invariant Subspace Property. He even showed that the space  $s$ does not satisfy the Invariant Subset Property~\cite{Golinski2}.  However, it is not known if the space  $H(\mathbb{C})$ does not satisfy the Invariant Subset Property. Since $H(\mathbb{C})$  and  $s$ are reflexive spaces, we remark that the reflexivity does not imply the Invariant Subspace Property in the case of non-normable Fréchet spaces.

We try in this paper to better understand which non-normable Fréchet spaces satisfy the (Hereditary) Invariant Subspaces/Subsets Properties. To this end, we first generalize the results obtained by Körber~\cite{Korber} and Shields~\cite{Shields} by showing that if $Y$ is a Fréchet space with a continuous norm then $\omega\oplus Y$ satisfies the Invariant Subspace Property. We then look at the approach of Goli\'{n}ski concerning the existence of operators without non-trivial invariant subspaces. In his paper~\cite{Golinski1}, Goli\'{n}ski states sufficient conditions for the construction of a Read-type operator without non-trivial invariant subspaces and he shows in \cite{Golinski2} that it is also possible to construct a Read-type operator without non-trivial invariant subset on $s$. In this paper, we state sufficient conditions for the construction of a Read-type operator without non-trivial invariant subset. While the conditions stated by Goli\'{n}ski concerned only Köthe sequence spaces of type $\ell_1$, our results can be applied to any Fréchet space with a Schauder basis and thus in particular to any Köthe sequence space. These conditions will allow us among other results to show that the space $H(\mathbb{C})$ does not satisfy the Invariant Subset Property and that a large family of non-normable Fréchet spaces do not satisfy the Hereditary Invariant Subset Property.

Our paper is organized as follows. In Section~\ref{Sec1}, we show that if the kernels of seminorms $p_j$ satisfy some conditions then $(X,(p_j)_j)$ satisfies the Invariant Subspace Property. In Section~\ref{Sec2}, we start by stating some facts concerning Schauder basis in Fréchet spaces (Subsection~\ref{Sec21}). We then state our sufficient conditions for the construction of a Read-type operator without non-trivial invariant subset and investigate the consequences of this result for the Invariant Subset Property (Subsection~\ref{Sec22}) and for the Hereditary Invariant Subset Property (Subsection~\ref{Sec23}). The technical proof of the above-mentioned sufficient conditions will be given in Section~\ref{last}.

\section{Fréchet spaces satisfying the (Hereditary) Invariant Subspace Property}\label{Sec1}

As mentioned in the Introduction, we know that there exist Banach spaces satisfying the Invariant Subspace Property~\cite{Argyros1} and even the Hereditary Invariant Subspace Property~\cite{Argyros2}. In the context of non-normable Fréchet spaces, it is easier to exhibit spaces with the Invariant Subspace Property by using the non-triviality of  kernels of seminorms inducing the topology of the space. 
\begin{theorem}\label{omega}
Let $(X,(p_j)_j)$ be a Fréchet space without continuous norm. If there exists $j_0$ such that for every $j\ge j_0$, $\ker p_{j+1}$ is a subspace of finite codimension in $\ker p_j$ then $X$ satisfies the Invariant Subspace Property.
\end{theorem}
\begin{proof}
Let $T$ be a continuous operator on $X$. Since $X$ does not possess a continuous norm, we can deduce that for every $j$, $\ker p_j$ is an infinite-dimensional subspace. Moreover, without loss of generality, we can assume that $\ker p_2$ is at least of codimension $2$ in $\ker p_1$, that $\ker p_{j+1}$ is at least of codimension $1$ in $\ker p_{j}$ for every $j\ge 1$ and that for every $j\ge 1$, there exists a constant $C_j>0$ such that for every $x\in X$, $p_j(Tx)\le C_j p_{j+1}(x)$.

Let $M_j$ be a complemented subspace of $\text{ker}\, p_{j+1}$ in $\text{ker}\, p_j$. Since $M_j$ is finite-dimensional, there exists a basis $(e^j_n)_{1\le n\le n_j}$ of $M_j$. The family $\{(e^j_n)_{1\le n\le n_j}\}_{j\ge 1}$ forms a Schauder basis of $\text{ker}\, p_1$. Indeed, if we consider $x\in \ker p_1$, we know that $x=\sum_{n=1}^{n_1}\alpha^{1}_n e^{1}_n+z_2$ where $z_2\in \text{ker}\, p_2$. Moreover, $z_2=\sum_{n=1}^{n_2}\alpha^{2}_n e^{2}_n+z_3$ where $z_3\in \text{ker}\, p_3$ and so on. We deduce that $x=\sum_{j=1}^{\infty }\sum_{n=1}^{n_j}\alpha^{j}_n e^{j}_n$ where the convergence is obvious. Moreover the decomposition is unique because if we have $\sum_{j=1}^{\infty }\sum_{n=1}^{n_j}(\alpha^{j}_n-\beta^{j}_n)e^j_n=0$ then
\[p_2\Big(\sum_{j=1}^{\infty }\sum_{n=1}^{n_j}(\alpha^{j}_n-\beta^{j}_n)e^{j}_n\Big)=p_2\Big(\sum_{n=1}^{n_1}(\alpha^{1}_n-\beta^{1}_n)e^1_n\Big)=0\]
and thus $\sum_{n=1}^{n_1}(\alpha^{1}_n-\beta^{1}_n)e^1_n\in \ker p_2\cap M_1=\{0\}$. Hence, we get $\alpha^1_n= \beta^1_n$ for every $1\le n\le n_1$. By repeating this argument, we obtain the uniqueness of the decomposition. We denote by $(u_n)_{n\ge 0}$ this Schauder basis. In particular, since $\ker p_2$ is at least of codimension $2$ in $\ker p_1$, we have $u_0$, $u_1\in M_1$.

Let $P_{M_1}:\ker p_1 \to M_1$ be the projection of $\ker\, p_1$ along $\ker p_2$ onto $M_1$.
Since $T(\ker p_{j+1})\subset \ker p_j$, we remark that for every $x\in \ker p_1$, if $P_{M_1}x=0$ then $x\in \ker p_2$ and thus $Tx\in \ker p_1$. We can thus let for every $l\ge 0$
\[A_{l}=\{n\in \mathbb{Z}_+: P_{M_1}T^lu_n\ne 0 \quad \text{and}\quad P_{M_1}T^{l'}u_n= 0\ \text{for every $0\le l'<l$}\}\] In particular, we have $A_0=\{n\in \mathbb{Z}_+:u_n\in M_1\}$ and each set $A_l$ is finite since for every $n \in \mathbb{Z}_+$,
\[p_2(T^lu_n)\le \left(\prod_{j=2}^{l+1}C_j\right) p_{l+2}(u_n)\]
and $\{n\in \mathbb{Z}_+: p_{l+2}(u_n)=0\}$ is cofinite.\\

There are therefore two possibilities:
\begin{enumerate}
\item If there exists $l_0\ge 1$ such that $A_l$ is empty for every $l\ge l_0$, then there exists a positive integer $n\notin\bigcup_{l\ge 0}A_l$ and by definition of sets $A_l$, we deduce that $P_{M_1}T^lu_n=0$ for every $l\ge 0$. Therefore, the subspace given by $\overline{\text{span}}\,\text{Orb}(u_n,T)$ is a non-trivial invariant subspace for $T$ since for every $x\in \text{Orb}(u_n,T)$, $P_{M_1}x=0$.

\item On the other hand, if there exists an increasing sequence $(l_m)_{m\ge 1}$ such that $A_{l_m}$ is a non-empty set for every $m\ge 1$, then we can show that there exists a non-zero continuous and linear map $\varphi:X\to \mathbb{K}$ such that the closed subspace $M:=\{x\in X: \varphi(T^lx)=0 \quad\text{for every $l\ge 0$}\}$ is a non-trivial invariant subspace. This is obvious that $M$ is invariant and that $M\ne X$. It remains to prove that $M\backslash\{0\}$ is non-empty for a good choice of $\varphi$. 

We first remark that there exists a sequence $(x^{(l)})_{l\ge 1}\subset \ker p_1$ such that for every $0\le l'<l$, $P_{M_1}T^{l'}x^{(l)}= 0$ and $P_{M_1}T^lx^{(l)}\ne 0$. Indeed, if $l\le l_{m}$ and $n\in A_{l_{m}}$, it suffices to consider $x^{(l)}=T^{l_{m}-l}u_n$. In particular, since $P_{M_1}x^{(l)}=0$, we know that $x^{(l)}_0=x^{(l)}_1=0$ where $x^{(l)}=\sum_{j=0}^{\infty}x^{(l)}_ju_j$. Moreover, we can assume that for every $j\ge 2$, $(x^{(l)}_j)_l$ is ultimately equal to~$0$.
Indeed, if $(x^{(l)}_2)_l$ is not ultimately equal to $0$, there exists an increasing sequence $(m_k)$ such that $x^{(m_k)}_2\ne 0$ and therefore for every $l\ge 1$, if $l< m_{k}$, it suffices to replace $x^{(l)}$ by $x^{(l)}-\frac{x^{(l)}_2}{x^{(m_{k})}_2}x^{(m_{k})}$. By repeating this argument, we can finally obtain a sequence $(x^{(l)})$ such that $(x^{(l)}_j)_l$ is ultimately equal to~$0$ for every $j\ge 2$ and such that for every $l'<l$, $P_{M_1}T^{l'}x^{(l)}= 0$ and $P_{M_1}T^lx^{(l)}\ne 0$. In particular, for every sequence $(\beta_l)_{l\ge 1}$, the series $\sum_{l\ge 1}\beta_l x^{(l)}$ belongs to $X$.

We now consider a map $\varphi:\ker p_1\to \mathbb{K}$ defined by \[\varphi\Big(\sum_{n\ge 0}x_n u_n\Big)= \sum_{n\in M_1} \alpha_n x_n\] where $\alpha_0$ is a non-zero scalar and $(\alpha_n)_{n\in M_1}$ are chosen such that $\varphi(T^lx^{(l)})\ne 0$ for every $l\ge 1$. Such a map exists since for every $l\ge 1$, we have $P_{M_1}T^lx^{(l)}\ne 0$. By the Hahn-Banach theorem, the linear map $\varphi$ can then be continuously extended to $X$ and we can construct by induction on $l$ a sequence $(\beta_l)_{l\ge 1}$ so that for every $l\ge 1$
\[\varphi\big(\frac{\alpha_1}{\alpha_0}T^lu_0-T^lu_1+\sum_{j=1}^l\beta_j T^lx^{(j)}\big)=0.\]

Therefore, the vector $x=\frac{\alpha_1}{\alpha_0}u_0-u_1+\sum_{j=1}^{\infty}\beta_jx^{(j)}$ (which is convergent) belongs to $M\backslash\{0\}$ since for every $l\ge 0$, $\varphi(T^lx^{(j)})=0$ for every $j>l$ and thus \[\varphi(T^lx)=\varphi(T^lu_0-\frac{\alpha_0}{\alpha_1}T^lu_1+\sum_{j=1}^{l}\beta_jT^lx^{(j)})=0.\]
\end{enumerate}
\end{proof}

The result of Körber~\cite{Korber} and Shields~\cite{Shields} concerning the space $\omega$ directly follows from Theorem~\ref{omega}. We can in fact remark that a Fréchet space without continuous norm $(X,(p_j))$ satisfies the assumptions of Theorem~\ref{omega} if and only if $X$ is isomorphic to $\omega\oplus Y$ where $Y$ is a Fréchet space with continuous norm. Indeed, it is clear that if $Y$ is a Fréchet space with continuous norm then $\omega\oplus Y$ satisfies the assumptions of Theorem~\ref{omega}. On the other hand, if $X$ is a Fréchet space without continuous norm satisfying the conditions of Theorem~\ref{omega} then $\ker p_{j_0}$ is isomorphic to $\omega$ (see \cite[Proposition 26.16]{Meise}) and thus $X$ is isomorphic to $\omega \oplus Y$ where $Y$ is a Fréchet space with continuous norm.

\begin{cor}
Let $Y$ be a Fréchet space with a continuous norm. Then $\omega\oplus Y$ satisfies the Invariant Subspace Property.
\end{cor}

We can then deduce from Theorem~\ref{omega} that the Invariant Subspace Property is not equivalent to the Hereditary Invariant Subspace Property and that the Invariant Subset Property is not equivalent to the Hereditary Invariant Subset Property.

\begin{cor}
There exist a separable infinite-dimensional Fréchet space satisfying the Invariant Subspace Property but no satisfying the Hereditary Invariant Subset Property.
\end{cor}
\begin{proof}
We deduce from Theorem~\ref{omega} that $\omega \oplus \ell_1$  satisfies the Invariant Subspace Property but does not satisfy the Hereditary Invariant Subset Property since $\ell_1$ does not satisfy the Invariant Subset Property \cite{Read3}.
\end{proof}

Thanks to Atzmon~\cite{Atzmon, Atzmon2} and Goli\'{n}ski \cite{Golinski1, Golinski2}, we know several examples of non-normable Fréchet spaces which do not satisfy the Invariant Subset Property or the Invariant Subspace Property. In the next section, we investigate sufficient conditions for having no non-trivial invariant subset by trying to take the best advantage of the non-normability of non-normable Fréchet spaces.

\section{Fréchet spaces not satisfying the (Hereditary) Invariant Subset Property}\label{Sec2}

In order to prove that a Fréchet space $X$ does not satisfy the Invariant Subset Property, we have to be able to construct an operator on $X$ for which each non-zero vector is hypercyclic, i.e. for which each non-zero vector $x\in X$ has a dense orbit. Our construction will be based on the construction done by Read for $\ell_1$ in \cite{Read3} and by Goli\'{n}ski for $s$ in \cite{Golinski2}. To this end, the existence of a Schauder basis satisfying suitable conditions will be necessary. We start by recalling some properties of Schauder basis in Fréchet spaces.

\subsection{Schauder basis in Fréchet spaces}\label{Sec21}

For a Banach space $(X,\|\cdot\|)$, it is well-known~(see \cite{Diestel}) that if $(e_n)_{n\ge 0}$ is a Schauder basis then the norm $ \vvvert\cdot \vvvert$ given by
\[ \vvvert x \vvvert=\sup_{N\ge 0}\Big\|\sum_{n=0}^{N}x_n e_n\Big\|\]
where $x=\sum_{n=0}^{\infty}x_n e_n$ is equivalent to $\|\cdot\|$. In particular, there exists a constant $C>0$ such that for every $M\le N$, for every $x_0,\dots,x_N\in \K$, 
\[\Big\|\sum_{n=0}^{M}x_n e_n\Big\|\le C\Big\|\sum_{n=0}^{N}x_n e_n\Big\|.\]
Indeed, if we let $x=\sum_{n=0}^{N}x_n e_n$, we have
\[\Big\|\sum_{n=0}^{M}x_n e_n\Big\|\le  \vvvert x \vvvert\le C \Big\|\sum_{n=0}^{N}x_n e_n\Big\|\]
where the last inequality comes from the equivalence between $ \vvvert \cdot \vvvert$ and $\|\cdot\|$.
This result can be adapted  to Schauder basis for Fréchet spaces.

\begin{theorem}[{\cite[Theorem 6 (p298)]{Jarchow}}]\label{Schauder}
Let $(X,(p_j)_j)$ be a Fréchet space with a Schauder basis $(e_n)_{n\ge 0}$. Then for every $j\ge 1$, there exist $C_j>0$ and $J\ge 1$ such that for every $M\le N$, for every $x_0,\dots,x_N\in \K$, 
\[p_j\Big(\sum_{n=0}^{M}x_n e_n\Big)\le C_j p_J\Big(\sum_{n=0}^{N}x_n e_n\Big).\]
\end{theorem}

In particular, if $(e_n)$ is a Schauder basis then the existence of a continuous norm can be expressed in terms of the basis $(e_n)$.

\begin{theorem}\label{Schauder2}
Let $(X,(p_j)_{j})$ be a Fréchet space with a Schauder basis $(e_n)_{n\ge 0}$.  Then $X$ possesses a continuous norm if and only if there exists $j\ge 1$ such that for every $n\ge 0$, $p_j(e_n)>0$,
\end{theorem}
\begin{proof}
 If $X$ possesses a continuous norm $\|\cdot\|$  then there exists $j\ge 1$ and $C>0$ such that for every $x\in X$, $\|x\|\le C p_j(x)$. In particular, we get $p_j(e_n)>0$ for every $n\ge 0$.\\
We now assume that for every $n$, $p_j(e_n)>0$ in order to prove the other implication. By Theorem~\ref{Schauder}, there exist $C_j>0$ and $J\ge 1$ such that for every $x=\sum_{n=0}^{\infty}x_n e_n\in X$, for every $k\ge 0$,
\[|x_k|=\frac{p_j(\sum_{n=0}^k x_n e_n- \sum_{n=0}^{k-1} x_n e_n)}{p_j(e_k)}\le 2C_j \frac{p_J(x)}{p_j(e_k)}.\]
Since for every $x\ne 0$, there exists $k$ such that $x_k\ne 0$, we deduce that $p_J(x)>0$ for every $x\in X\backslash\{0\}$. In other words, $p_J$ is a continuous norm.
\end{proof}

\subsection{Invariant Subset Property}\label{Sec22}
In Section~\ref{Sec1}, we have mainly used the non-triviality of kernels of seminorms in order to deduce the existence of non-trivial invariant subspaces. In this section, we will use another particularity of non-normable Fréchet spaces concerning the continuity of operators. Indeed, if we consider a linear map $T$ on a Fréchet space $(X,(p_j)_{j})$, then $T$ is continuous if and only for every $j\ge 1$, there exists $C_j$ and $J\ge 1$ such that for every $x\in X$,
\[p_j(Tx)\le C_j p_J(x).\]
The main difference with Banach spaces relies on the fact that the above inequalities can involve two seminorms $p_j$ and $p_J$ which are not equivalent. In particular, it is possible that $T$ is continuous and that for every $R>0$, there exists $x$ such that 
\[p_j(Tx)> R p_j(x).\]
For instance, if we consider $X=H(\C)$ endowed with the norms $p_j(f)=\sup_{|z|\le j}|f(z)|$ and the derivative operator $D$, we have $p_j(Df)\le p_{j+1}(f)$ for every $f\in H(\C)$ and $p_j(Dz^n)=nj^{n-1}=\frac{n}{j}p_j(z^n)$ for every $n\ge 1$.

This particularity of non-normable Fréchet spaces is not used in \cite{Golinski1} since the considered operators satisfy $p_j(Tx)\le C_j p_j(x)$ for every $j\ge 1$. Due to this fact, Goli\'{n}ski can only consider Fréchet spaces with $\ell_1$-norms.  We investigate how this particularity can be used to construct operators without non-trivial invariant subset by only requiring conditions on the elements $p_j(e_n)$.

Given a finite family $(j,m)_{1\le j\le J, 1\le m\le M}$, we denote by $\next(j,m)$ the next element of the family $(j,m)_{j\le J, m\le M}$ endowed with the lexicographical order.

\begin{theorem}\label{technique}
Let $(X,(p_j)_{j})$ be a Fréchet space possessing a Schauder basis $(e_n)_{n\ge 0}$.
If $p_1(e_n)>0$ for every $n$ and if for every $\varepsilon>0$, every $C,M,N\ge 1$, every $K\ge J\ge 2$, there exist $(n_{j,m})_{j\le J, m\le M}\subset]N,\infty[$ with $n_{j,m}\ne n_{j',m'}$ for every $(j,m)\ne (j',m')$ and a sequence $(\alpha_{j,m})_{j\le J, m\le M}$ of non-zero scalars such that for every $j \le J$, every $m\le M$,
\begin{enumerate}
\item $p_{J-1}(\alpha_{1,m}e_{n_{1,m}})\le \varepsilon$;
\item $p_1(\alpha_{J,m}e_{n_{J,m}})\ge \frac{1}{\varepsilon}$;
\item for every $1\le l\le M$, if $\next^l(j,m)$ exists then for every $k< K$,
\[C p_k(\alpha_{\next^{l}(j,m)}e_{n_{\next^{l}(j,m)}})\le p_{k+1}(\alpha_{j,m}e_{n_{j,m}});\]
\end{enumerate}
then $X$ does not satisfy the Invariant Subset Property.
\end{theorem}

The proof of this theorem will be given in Section~\ref{last}. We focus here on the consequences of this result. In view of Theorem~\ref{Schauder2}, Theorem~\ref{technique} cannot be applied to Fréchet spaces without continuous norm. We can  also remark that if $X$ is a Banach space then the assumptions of Theorem~\ref{technique} are never satisfied. Indeed, we deduce from $(2)$ and $(3)$ that for every $m\le M$, every $j\le J$, 
\begin{equation*}
p_{J-j+1}(\alpha_{j,m} e_{n_{j,m}})\ge \frac{1}{\varepsilon}.
\end{equation*}
In particular, we get $p_{J}(\alpha_{1,1}e_{n_{1,1}})\ge \frac{1}{\varepsilon}$. Therefore, since $p_{J-1}(\alpha_{1,1}e_{n_{1,1}})\le \varepsilon$ by (1), we conclude that for every $J\ge 2$, the seminorms $p_J$ and $p_{J-1}$ are not equivalent and thus $X$ cannot be a Banach space.

Although Banach spaces cannot satisfy Theorem~\ref{technique}, we will see that a lot of Fréchet spaces with a continuous norm satisfy this theorem. In order to satisfy the assumptions of Theorem~\ref{technique}, it seems to be necessary to be able to find for every $j$, some elements $e_n$ such that the ratio between $p_{j+1}(e_n)$ and $p_{j}(e_n)$ is as large as desired and for which we have some control on the values of $p_l(e_n)$ for $l\le j$. The following corollary is based on this idea and give us new examples of Fréchet spaces with a continuous norm which do not satisfy the Invariant Subset Property.

\begin{cor}\label{cor gen}
Let $(X,(p_j)_{j})$ be a Fréchet space possessing a Schauder basis~$(e_n)_{n\ge 0}$. If $p_1(e_n)>0$ for every $n$ and if for every $j\ge 1$, there exists $C_j\ge 1$ such that for every $L,N\ge 0$, there exists $n\ge N$ such that
\[p_j(e_n) \le C_j p_1(e_n) \quad \text{and}\quad p_{j+1}(e_n)\ge L p_j(e_n),\]
then $X$ does not satisfy the Invariant Subset Property.
\end{cor}
\begin{proof}
Let $\varepsilon>0$, $C,M,N\ge 1$, $K\ge J\ge 2$ and $\Gamma=\max\{C_j:1\le j\le K\}$. We select different indices $(n_{J,m})_{m\le M}\subset ]N,\infty[$ and fix $\alpha_{J,m}>0$ for every $ m \le M$ such that 
\[p_1(\alpha_{J,M}e_{n_{J,M}})\ge \frac{1}{\varepsilon}\quad\text{and for every $m<M$}\quad p_1(\alpha_{J,m}e_{n_{J,m}})= C p_{K}(\alpha_{J,m+1}e_{n_{J,m+1}}).\]
Let $\eta=\frac{\varepsilon}{(\Gamma C)^{(J-2)M} C^{M-1}}$. We can then find by assumption  $(n_{J-1,m})_{m\le M}\subset ]\max_m\{n_{J,m}\},\infty[$ such that 
\[p_{2}(e_{n_{J-1,M}})>\frac{Cp_K(\alpha_{J,1}e_{n_{J,1}})}{\eta}p_1(e_{n_{J-1,M}})\]
and such that for every $m<M$
\[p_{2}(e_{n_{J-1,m}})> \frac{p_K(e_{n_{\next(J-1,m)}})}{p_1(e_{n_{\next(J-1,m)}})} p_1(e_{n_{J-1,m}}).\]
We then fix $(\alpha_{J-1,m})_{m\le M}$ such that 
\[p_1(\alpha_{J-1,M}e_{n_{J-1,M}})=\eta\]
and such that for every $m<M$
\[p_1(\alpha_{J-1,m}e_{n_{J-1,m}})=C p_{1}(\alpha_{\next(J-1,m)}e_{n_{\next(J-1,m)}}).\]
Let $1\le j<J-1$. If $(n_{j+1,m})_{m\le M}$ has been chosen, we choose $(n_{j,m})_{m\le M}\subset ]\max_m\{n_{j+1,m}\},\infty[$ such that for every $m\le M$
\[p_{J-j+1}(e_{n_{j,m}})> \frac{p_K(e_{n_{\next(j,m)}})}{p_1(e_{n_{\next(j,m)}})} p_{J-j}(e_{n_{j,m}})\quad \text{and}\quad p_{J-j}(e_{n_{j,m}}) < \Gamma p_{1}(e_{n_{j,m}})\]
and we then fix $(\alpha_{j,m})_{m\le M}$ such that
\[p_1(\alpha_{j,M}e_{n_{j,M}})= C p_{J-j-1}(\alpha_{j+1,1}e_{n_{j+1,1}}).\]
and such that for every $m<M$
\[p_1(\alpha_{j,m}e_{n_{j,m}})= C p_{J-j}(\alpha_{\next(j,m)}e_{n_{\next(j,m)}}).\]
We can then show that the family  $\{\alpha_{j,m}e_{n_{j,m}}\}_{j\le J,m\le M}$ satisfies the assumptions of Theorem~\ref{technique}. Indeed,
by construction, we have $p_1(\alpha_{n_{J,m}}e_{n_{J,m}})\ge \frac{1}{\varepsilon}$ for every $m\le M$ since $p_1(\alpha_{J,M}e_{n_{J,M}})\ge \frac{1}{\varepsilon}$ and since for every $m<M$,
\[p_1(\alpha_{n_{J,m}}e_{n_{J,m}})=C p_{K}(\alpha_{J,m+1}e_{n_{J,m+1}})\ge p_1(\alpha_{J,m+1}e_{n_{J,m+1}}).\]

We can also show that for every $m\le M$, we have $p_{J-1}(\alpha_{1,m}e_{n_{1,m}})\le \varepsilon$. Let $m\le M$ and $j< J-1$. If $m<M$, we get
\[p_{J-j}(\alpha_{j,m}e_{n_{j,m}})\le \Gamma p_1(\alpha_{j,m}e_{n_{j,m}})= \Gamma C p_{J-j}(\alpha_{\next(j,m)}e_{n_{\next(j,m)}})\]
and if $m=M$, we get
\[p_{J-j}(\alpha_{j,M}e_{n_{j,M}})\le \Gamma p_1(\alpha_{j,M}e_{n_{j,M}}) = \Gamma  C p_{J-j-1}(\alpha_{j+1,1}e_{n_{j+1,1}}).\]
Therefore, for every $m\le M$, we deduce that
\begin{align*}
p_{J-1}(\alpha_{1,m}e_{n_{1,m}})&\le (\Gamma C)^{(J-3)M+M-m+1}p_{1}(\alpha_{J-1,1}e_{n_{J-1,1}})\\
&\le (\Gamma C)^{(J-3)M+M-m+1} C^{M-1}p_{1}(\alpha_{J-1,M}e_{n_{J-1,M}})\\
&\le \varepsilon.
\end{align*}

It remains to show that  for every $l\le M$, every $(j,m)$, every $1 \le k< K$ if $\next^l(j,m)$ exists then
\begin{align*}
p_{k+1}(\alpha_{j,m}e_{n_{j,m}})&\ge C p_{k}(\alpha_{\next^l(j,m)}e_{n_{\next^l(j,m)}}).
\end{align*}
We remark that for every $(j,m)\ne (J,M)$, for every $1 \le k\le K$,
\begin{enumerate}
\item if $k\le J-j$ and $m\ne M$, we have
\begin{align*}
p_{k}(\alpha_{j,m}e_{n_{j,m}})&\ge p_{1}(\alpha_{j,m}e_{n_{j,m}})\\
&\ge Cp_{J-j}(\alpha_{\next(j,m)}e_{n_{\next(j,m)}});
\end{align*}
\item if $k\le J-j$, $j\le J-2$ and $m=M$, we have
\begin{align*}
p_{k}(\alpha_{j,M}e_{n_{j,M}})&\ge p_{1}(\alpha_{j,M}e_{n_{j,M}})\\
&\ge  C p_{J-j-1}(\alpha_{j+1,1}e_{n_{j+1,1}});
\end{align*}
\item if $k> J-j$, we have
\[p_{k}(\alpha_{j,m}e_{n_{j,m}}) \ge C p_K(\alpha_{\next(j,m)}e_{n_{\next(j,m)}}).\]
Indeed, if $j\le J-2$, we have
\begin{align*}
p_{k}(\alpha_{j,m}e_{n_{j,m}})&\ge p_{J-j+1}(\alpha_{j,m}e_{n_{j,m}})\\
&\ge \frac{p_K(e_{n_{\next(j,m)}})}{p_1(e_{n_{\next(j,m)}})} p_{J-j}(\alpha_{j,m}e_{n_{j,m}})\\
&\ge C \frac{p_K(e_{n_{\next(j,m)}})}{p_1(e_{n_{\next(j,m)}})}p_{J-j-1}(\alpha_{\next(j,m)}e_{n_{\next(j,m)}})\\
&\ge C p_K(\alpha_{\next(j,m)}e_{n_{\next(j,m)}}).
\end{align*}
If  $j=J-1$ and $m\ne M$, we have
\begin{align*}
p_{k}(\alpha_{j,m}e_{n_{j,m}})&\ge p_{2}(\alpha_{j,m}e_{n_{j,m}})\\
&\ge \frac{p_K(e_{n_{\next(j,m)}})}{p_1(e_{n_{\next(j,m)}})} p_{1}(\alpha_{j,m}e_{n_{j,m}})\\
&\ge C \frac{p_K(e_{n_{\next(j,m)}})}{p_1(e_{n_{\next(j,m)}})}p_{1}(\alpha_{\next(j,m)}e_{n_{\next(j,m)}})\\
&\ge C p_K(\alpha_{\next(j,m)}e_{n_{\next(j,m)}}).
\end{align*}
If  $j=J-1$ and $m= M$, we have
\begin{align*}
p_{k}(\alpha_{J-1,M}e_{n_{J-1,M}})&\ge p_{2}(\alpha_{J-1,M}e_{n_{J-1,M}})\\
&\ge \frac{C p_K(\alpha_{J,1}e_{n_{J,1}})}{\eta} p_{1}(\alpha_{J-1,M}e_{n_{J-1,M}})\\
&\ge C p_K(\alpha_{J,1}e_{n_{J,1}})
\end{align*}
Finally, if $j=J$ and $m\ne M$, we have
\begin{align*}
p_{k}(\alpha_{J,m}e_{n_{J,m}})&\ge p_{1}(\alpha_{J,m}e_{n_{J,m}})\\
&\ge C p_{K}(\alpha_{\next(J,m)}e_{n_{\next(J,m)}}).
\end{align*}
\end{enumerate}
We deduce that for every $l\le M$, every $(j,m)$, every $1 \le k< K$ if $\next^l(j,m)$ exists then
\begin{align*}
p_{k+1}(\alpha_{j,m}e_{n_{j,m}})&\ge C p_{k}(\alpha_{\next^l(j,m)}e_{n_{\next^l(j,m)}})
\end{align*}
because
\begin{enumerate}
\item if $k<J-j$ and $\next^l(j,m)=(j,m')$ then
\begin{align*}
p_{k+1}(\alpha_{j,m}e_{n_{j,m}})&\ge C p_{J-j}(\alpha_{\next(j,m)}e_{n_{\next(j,m)}})\\
&\ge C^l p_{J-j}(\alpha_{\next^l(j,m)}e_{n_{\next^l(j,m)}})\\
&\ge C p_{k}(\alpha_{\next^l(j,m)}e_{n_{\next^l(j,m)}});
\end{align*}
\item if $k<J-j$ and $\next^l(j,m)=(j+1,m')$ then
\begin{align*}
p_{k+1}(\alpha_{j,m}e_{n_{j,m}})&\ge C^{M-m}p_{J-j}(\alpha_{j,M}e_{n_{j,M}})\\
&\ge C^{M-m+1} p_{J-j-1}(\alpha_{j+1,1}e_{n_{j+1,1}})\\
&\ge C^l p_{J-j-1}(\alpha_{\next^l(j,m)}e_{n_{\next^l(j,m)}})\\
&\ge C p_{k}(\alpha_{\next^l(j,m)}e_{n_{\next^l(j,m)}});
\end{align*}
\item if $k\ge J-j$, we have
\begin{align*}
p_{k+1}(\alpha_{j,m}e_{n_{j,m}})&\ge C p_K(\alpha_{\next(j,m)}e_{n_{\next(j,m)}})\\
&\ge C^l p_K(\alpha_{\next^l(j,m)}e_{n_{\next^l(j,m)}})\\
&\ge C p_k(\alpha_{\next^l(j,m)}e_{n_{\next^l(j,m)}}).
\end{align*}
\end{enumerate}
\end{proof}

An interesting family of Fréchet spaces with a Schauder basis is given by Köthe sequence spaces.

\begin{definition}
Let $A=(a_{j,k})_{j\ge 1,k\ge 0}$ be a matrix such that for every $k$, there exists $j$ satisfying $a_{j,k}>0$ and such that for every $k$, the sequence $(a_{j,k})_j$ is non-decreasing.
The Köthe sequence space $\lambda^p(A)$ with $1\le p<\infty$ is given by
\begin{align*}
\lambda^p(A)=\Big\{(x_k)_k\in \mathbb{K}^{\mathbb{N}} : p_j((x_k)_k)=\Big(\sum_{k=0}^{\infty}|x_ka_{j,k}|^p\Big)^{\frac 1 p}<\infty, \ j\ge 1\Big\}
\end{align*}
and the Köthe sequence space $c_0(A)$ is given by
\[c_0(A)=\{(x_k)_k\in \mathbb{K}^{\mathbb{N}}: \lim_{k\rightarrow \infty}|x_k|a_{j,k}=0, \ j\ge 1\} \text{ with } p_j((x_k)_k)= \max_k|x_k|a_{j,k}.\]
\end{definition}

\begin{example}
Let $(q_j)_{j\ge 1}$ be the increasing enumeration of prime numbers and $n\ge 0$. If $\prod_{l=1}^{\infty}q_l^{k_l}$ is the unique prime factorization of $n+1$ then we let $a_{j,n}=\prod_{l=1}^jq_l^{k_l}$. We can deduce from Corollary~\ref{cor gen} that the Fréchet space $X=\lambda^p(A)$ with $1\le p<\infty$ or $c_0(A)$ does not satisfy the Invariant Subset Property since for every $j\ge 1$, every $k\ge 1$, we have
\[ p_j(e_{q_{j+1}^k-1})=1= p_1(e_{q_{j+1}^k-1}) \quad \text{and}\quad p_{j+1}(e_{q_{j+1}^k-1})=q_{j+1}^k.\]
This space is an example of non-reflexive and non-normable Fréchet space without the Invariant Subset Property.
\end{example}

In \cite{Golinski2}, Goli\'{n}ski proved that $s=\lambda^1(A)$ (with $a_{j,k}=(k+1)^j$) satisfies the Invariant Subset property. However, this space does not satisfy the assumptions of Corollary~\ref{cor gen}. One can therefore wonder if the space $s$ satisfies the assumptions of Theorem~\ref{technique}. Actually a lot of natural Fréchet spaces, such as $s$ or $H(\mathbb{C})$, satisfy the assumptions of Theorem~\ref{technique} and thus do not satisfy the Invariant Subset Property. These examples will be obtained thanks to the following corollary.

\begin{cor}\label{cor Kot}
Let $(X,(p_j)_{j})$ be a Fréchet space possessing a Schauder basis~$(e_n)_{n\ge 0}$ such that $p_1(e_n)>0$ for every $n$ and let $R_{j,n}=\frac{p_{j+1}(e_n)}{p_{j}(e_n)}$ for every $j\ge 1$, $n\ge 0$. If 
\begin{itemize}
\item $(R_{j,n})_j$ and $(R_{j,n})_n$ are non-decreasing sequences such that for every $j\ge 1$, $(R_{j,n})_n$ tends to infinity,
\item for every $j\ge 1$, every $n\ge 0$, 
$R_{j+1,n}\le R_{j,n+1}$,
\item for every $j\ge 1$
\[0<\inf_n \frac{p_j(e_n)}{p_{j}(e_{n+1})}\le \sup_n \frac{p_j(e_n)}{p_{j}(e_{n+1})}<\infty,\]
\end{itemize}
then $X$ does not satisfy the Invariant Subset Property.
\end{cor}
\begin{proof}
Let $\varepsilon>0$, $C,M,N\ge 1$ and $K\ge J\ge 2$. There exist $\gamma<1<\Gamma$ such that for every $j\le K$, every $n\ge 0$,
\[\gamma<\frac{p_j(e_n)}{p_{j}(e_{n+1})}<\Gamma.\]
In particular, for every $k\ge 1$, every $j\le K$ every $n\ge 0$, we have
\[\gamma^k<\frac{p_j(e_n)}{p_{j}(e_{n+k})}<\Gamma^k\]
and thus for every $j<K$
\[R_{j,n}=\frac{p_{j+1}(e_n)}{p_j(e_n)}\ge  \frac{\gamma^k}{\Gamma^k}\frac{p_{j+1}(e_{n+k})}{p_j(e_{n+k})}=\frac{\gamma^k}{\Gamma^k}R_{j,n+k}.\]
We then consider $N_0$ such that for every $j\le K$, every $n\ge N_0$,
\[R_{j,n}> \frac{C^{J^2M}}{\varepsilon^2} \Big(\frac{\Gamma}{\gamma}\Big)^{J^2M}\]
and we let $n_{j,m}=N_0+(j-1)M+m-1$. We also let for every $m\le M$
\[\alpha_{1,m}=\frac{\varepsilon\prod_{j=J-1}^{K-2}R_{j,n_{1,1}}}{C^{m-1}p_{K-1}(e_{n_{1,m}})}\]
and for every $1<j\le J$
\[\alpha_{j,m}=\frac{p_K(\alpha_{j-1,1}e_{n_{j-1,1}})}{C^{(j-1)M+m-1}p_{K-1}(e_{n_{j,m}})}\]
Let $j \le J$ and $m\le M$.
\begin{enumerate}
\item We have 
\[p_{J-1}(\alpha_{1,m}e_{n_{1,m}})=\frac{\varepsilon \prod_{j=J-1}^{K-2}R_{j,n_{1,1}}}{C^{m-1}p_{K-1}(e_{n_{1,m}})}p_{J-1}(e_{n_{1,m}})= \frac{\varepsilon \prod_{j=J-1}^{K-2}R_{j,n_{1,1}}}{C^{m-1}\prod_{j=J-1}^{K-2}R_{j,n_{1,m}}}\le \frac{\varepsilon}{C^{m-1}}\le \varepsilon.\]
\item We have
\begin{align*}
p_1(\alpha_{J,m}e_{n_{J,m}})&=\frac{p_K(\alpha_{J-1,1}e_{n_{J-1,1}})}{C^{(J-1)M+m-1}p_{K-1}(e_{n_{J,m}})}p_1(e_{n_{J,m}})\\
&\ge \frac{1}{C^{JM}}\frac{p_K(\alpha_{J-1,1}e_{n_{J-1,1}})p_1(e_{n_{J,M}})}{p_{K-1}(e_{n_{J,M}})}\\
&=\frac{\prod_{j=J-1}^{K-2}R_{j,n_{1,1}}}{C^{JM}} \frac{\varepsilon p_K(e_{n_{1,1}})}{p_{K-1}(e_{n_{1,1}})}\Big(\prod_{i=2}^{J-1} \frac{p_K(e_{n_{i,1}})}{C^{(i-1)M}p_{K-1}(e_{n_{i,1}})}\Big)   \frac{p_1(e_{n_{J,M}})}{p_{K-1}(e_{n_{J,M}})}.
\end{align*}
We remark that 
\[\frac{p_K(e_{n_{1,1}})}{p_{K-1}(e_{n_{1,1}})}=R_{K-1,n_{1,1}}\ge \frac{\gamma^{n_{J,M}-n_{1,1}}}{\Gamma^{n_{J,M}-n_{1,1}}}R_{{K-1},n_{J,M}},\]
that 
\[\prod_{i=2}^{J-1} \frac{p_K(e_{n_{i,1}})}{p_{K-1}(e_{n_{i,1}})}=\prod_{i=2}^{J-1} R_{K-1,n_{i,1}}\ge \prod_{i=2}^{J-1} \frac{\gamma^{n_{J,M}-n_{i,1}}}{\Gamma^{n_{J,M}-n_{i,1}}} R_{K-1,n_{J,M}},\]
and that 
\[\frac{p_1(e_{n_{J,M}})}{p_{K-1}(e_{n_{J,M}})}= \prod_{i=1}^{K-2}\frac{1}{R_{i,n_{J,M}}}.\]
We conclude that
\begin{align*}
&p_1(\alpha_{J,m}e_{n_{J,m}})\\
&\quad\ge \frac{\varepsilon}{C^{J^2M}} \Big(\frac{\gamma^{J-1}}{\Gamma^{J-1}}\Big)^{n_{J,M}-n_{1,1}}(R_{K-1},n_{J,M})^{J-1}\Big(\prod_{j=J-1}^{K-2}R_{j,n_{1,1}}\Big)\prod_{i=1}^{K-2}\frac{1}{R_{i,n_{J,M}}}\\
&\quad\ge \frac{\varepsilon}{C^{J^2M}} \Big(\frac{\gamma^{J-1}}{\Gamma^{J-1}}\Big)^{n_{J,M}-n_{1,1}}(R_{K-1},n_{J,M})^{J-1} \Big(\frac{\gamma}{\Gamma}\Big)^{n_{J,M}-n_{1,1}}\Big(\prod_{j=J-1}^{K-2}R_{j,n_{J,M}}\Big)\prod_{i=1}^{K-2}\frac{1}{R_{i,n_{J,M}}}\\
&\quad\ge \frac{\varepsilon}{C^{J^2M}} \Big(\frac{\gamma^{J}}{\Gamma^{J}}\Big)^{n_{J,M}-n_{1,1}}R_{K-1},n_{J,M}\quad \text{since $(R_{j,n_{J,M}})_j$ is non-decreasing}\\
&\quad>\frac{1}{\varepsilon}
\end{align*}
\item Let $l\le M$ such that $\next^l(j,m)$ exists and $k< K$. We have several possibilities:
\begin{enumerate}
\item If $\next^l(j,m)=(j,m')$ and $j=1$ then
\begin{align*}
p_k(\alpha_{\next^{l}(j,m)}e_{n_{\next^{l}(j,m)}})&= \frac{p_{K-1}(\alpha_{\next^{l}(j,m)}e_{n_{\next^{l}(j,m)}})}{\prod_{i=k}^{K-2}R_{i,n_{\next^{l}(j,m)}}}\\
&=\frac{\varepsilon \prod_{j=J-1}^{K-2}R_{j,n_{1,1}}}{C^{m'-1}\prod_{i=k}^{K-2}R_{i,n_{\next^{l}(j,m)}}}
\end{align*}
and thus
\begin{align*}
\frac{p_k(\alpha_{\next^{l}(j,m)}e_{n_{\next^{l}(j,m)}})}{p_{k+1}(\alpha_{j,m}e_{n_{j,m}})}&\le \frac{p_k(\alpha_{\next^{l}(j,m)}e_{n_{\next^{l}(j,m)}})}{p_{k}(\alpha_{j,m}e_{n_{j,m}})}\\
&= \frac{\prod_{i=k}^{K-2}R_{i,n_{j,m}}}{C^{m'-m}\prod_{i=k}^{K-2}R_{i,n_{\next^{l}(j,m)}}}\\
&\le \frac{1}{C^{m'-m}}=\frac{1}{C^{l}} \quad \text{since $(R_{i,n})_n$ is non-decreasing}.
\end{align*}
\item If $\next^l(j,m)=(j,m')$ and $j>1$ then
\begin{align*}
p_k(\alpha_{\next^{l}(j,m)}e_{n_{\next^{l}(j,m)}})&= \frac{p_{K-1}(\alpha_{\next^{l}(j,m)}e_{n_{\next^{l}(j,m)}})}{\prod_{i=k}^{K-2}R_{i,n_{\next^{l}(j,m)}}}\\
&=\frac{p_K(\alpha_{j-1,1}e_{n_{j-1,1}})}{C^{(j-1)M+m'-1}\prod_{i=k}^{K-2}R_{i,n_{\next^{l}(j,m)}}}
\end{align*}
and thus
\begin{align*}
\frac{p_k(\alpha_{\next^{l}(j,m)}e_{n_{\next^{l}(j,m)}})}{p_{k+1}(\alpha_{j,m}e_{n_{j,m}})}&\le \frac{p_k(\alpha_{\next^{l}(j,m)}e_{n_{\next^{l}(j,m)}})}{p_{k}(\alpha_{j,m}e_{n_{j,m}})}\\
&= \frac{\prod_{i=k}^{K-2}R_{i,n_{j,m}}}{C^{m'-m}\prod_{i=k}^{K-2}R_{i,n_{\next^{l}(j,m)}}}\\
&\le \frac{1}{C^{m'-m}}=\frac{1}{C^{l}}.
\end{align*}
\item If $\next^l(j,m)=(j+1,m')$ and $j=1$ then 
\begin{align*}
p_k(\alpha_{\next^{l}(j,m)}e_{n_{\next^{l}(j,m)}})&=\frac{p_K(\alpha_{1,1}e_{n_{1,1}})}{C^{jM+m'-1}\prod_{i=k}^{K-2}R_{i,n_{\next^{l}(j,m)}}}\\
&=\frac{\varepsilon p_K(e_{n_{1,1}})\prod_{j=J-1}^{K-2}R_{j,n_{1,1}}}{C^{jM+m'-1}p_{K-1}(e_{n_{1,1}})\prod_{i=k}^{K-2}R_{i,n_{\next^{l}(j,m)}}}.
\end{align*}
Moreover, if $k<K-1$, we have
\[p_{k+1}(\alpha_{j,m}e_{n_{j,m}})= \frac{\varepsilon \prod_{j=J-1}^{K-2}R_{j,n_{1,1}}}{C^{m-1}\prod_{i=k+1}^{K-2}R_{i,n_{j,m}}}\]
 and thus
\begin{align*}
\frac{p_k(\alpha_{\next^{l}(j,m)}e_{n_{\next^{l}(j,m)}})}{p_{k+1}(\alpha_{j,m}e_{n_{j,m}})}&=
\frac{p_K(e_{n_{1,1}})\prod_{i=k+1}^{K-2}R_{i,n_{j,m}}}{C^{jM+m'-m} p_{K-1}(e_{n_{1,1}})\prod_{i=k}^{K-2}R_{i,n_{\next^{l}(j,m)}}}\\
&=\frac{R_{K-1,n_{1,1}}\prod_{i=k+1}^{K-2}R_{i,n_{j,m}}}{C^{jM+m'-m}\prod_{i=k}^{K-2}R_{i,n_{\next^{l}(j,m)}}}\le \frac{1}{C^{jM+m'-m}}=\frac{1}{C^l}
\end{align*}
because $R_{i+1,n_{j,m}}\le R_{i,n_{j,m}+1}\le R_{i,n_{\next^{l}(j,m)}}$.

On the other hand, if $k=K-1$, we have $p_{k+1}(\alpha_{j,m}e_{n_{j,m}})= \frac{\varepsilon R_{K-1,n_{j,m}}\prod_{j=J-1}^{K-2}R_{j,n_{1,1}}}{C^{m-1}}$ and thus
\begin{align*}
\frac{p_k(\alpha_{\next^{l}(j,m)}e_{n_{\next^{l}(j,m)}})}{p_{k+1}(\alpha_{j,m}e_{n_{j,m}})}&=
\frac{R_{K-1,n_{1,1}}}{C^{jM+m'-m}R_{K-1,n_{j,m}}}\\
&\le \frac{1}{C^{jM+m'-m}}=\frac{1}{C^l}.
\end{align*}
\item If $\next^l(j,m)=(j+1,m')$ and $j>1$ then 
\begin{align*}
p_k(\alpha_{\next^{l}(j,m)}e_{n_{\next^{l}(j,m)}})&=\frac{p_K(\alpha_{j,1}e_{n_{j,1}})}{C^{jM+m'-1}\prod_{i=k}^{K-2}R_{i,n_{\next^{l}(j,m)}}}\\
&=\frac{p_{K-1}(\alpha_{j,1}e_{n_{j,1}})R_{K-1,n_{j,1}}}{C^{jM+m'-1}\prod_{i=k}^{K-2}R_{i,n_{\next^{l}(j,m)}}}\\
&=\frac{p_K(\alpha_{j-1,1}e_{n_{j-1},1})R_{K-1,n_{j,1}}}{C^{jM+m'-1}C^{(j-1)M}\prod_{i=k}^{K-2}R_{i,n_{\next^{l}(j,m)}}}.
\end{align*}
If $k<K-1$, we also have
\[p_{k+1}(\alpha_{j,m}e_{n_{j,m}})=\frac{p_K(\alpha_{j-1,1}e_{n_{j-1,1}})}{C^{(j-1)M+m-1}\prod_{i=k+1}^{K-2}R_{i,n_{j,m}}}\]
 and thus
\begin{align*}
\frac{p_k(\alpha_{\next^{l}(j,m)}e_{n_{\next^{l}(j,m)}})}{p_{k+1}(\alpha_{j,m}e_{n_{j,m}})}&= 
\frac{R_{K-1,n_{j,1}}\prod_{i=k+1}^{K-2}R_{i,n_{j,m}}}{C^{jM+m'-m} \prod_{i=k}^{K-2}R_{i,n_{\next^{l}(j,m)}}}\le \frac{1}{C^{l}}.
\end{align*}
Finally, if $k=K-1$, we have $p_{k+1}(\alpha_{j,m}e_{n_{j,m}})=\frac{p_K(\alpha_{j-1,1}e_{n_{j-1,1}})R_{K-1,n_{j,m}}}{C^{(j-1)M+m-1}}$ and thus 
\begin{align*}
\frac{p_k(\alpha_{\next^{l}(j,m)}e_{n_{\next^{l}(j,m)}})}{p_{k+1}(\alpha_{j,m}e_{n_{j,m}})}&= 
\frac{R_{K-1,n_{j,1}}}{C^{jM+m'-m} R_{K-1,n_{j,m}}}\le \frac{1}{C^{l}}.
\end{align*}
\end{enumerate}
\end{enumerate}

The desired result then follows from Theorem~\ref{technique}.
\end{proof}

We can deduce the result of Goli\'{n}ski~\cite{Golinski2} concerning the space $s$ thanks to this corollary.

\begin{example}
The space $s$ can be defined as the Köthe sequence space $\lambda^1(A)$ with $a_{j,n}=(n+1)^j$. We then get
\[R_{j,n}=\frac{p_{j+1}(e_n)}{p_j(e_n)}=\frac{(n+1)^{j+1}}{(n+1)^j}=n+1.\]
Therefore, the sequences $(R_{j,n})_j$ and $(R_{j,n})_n$ are non-decreasing and $\lim_n R_{j,n}=\infty$. Moreover, we have $R_{j+1,n}\le R_{j,n+1}$ and since $\frac{p_j(e_n)}{p_j(e_{n+1})}=\frac{(n+1)^j}{(n+2)^j}$, we have
\[\frac{1}{2^j}<\inf_n \frac{p_j(e_n)}{p_j(e_{n+1})}\le \sup_n \frac{p_j(e_n)}{p_j(e_{n+1})}\le 1.\]
It follows from Corollary~\ref{cor Kot} that $s$ does not satisfy the Invariant Subset Property.

\end{example}

In \cite{Golinski1}, it is also shown that $H(\mathbb{C})$ does not satisfy the Invariant Subspace Property. Thanks to Corollary~\ref{cor Kot}, we improve this result by showing that $H(\mathbb{C})$ does not satisfy the Invariant Subset Property.

\begin{cor}
$H(\mathbb{C})$ does not satisfy the Invariant Subset Property.
\end{cor}
\begin{proof}
We can identify $H(\mathbb{C})$ with $\lambda^{1}(A)$ where $a_{j,n}=(2^j)^{n}$. We then get 
\[R_{j,n}=\frac{p_{j+1}(e_n)}{p_{j}(e_{n})}=\frac{(2^{j+1})^n}{(2^j)^n}=2^n\] which is non-decreasing in $j$ and non-decreasing in $n$. Moreover, $(R_{j,n})_n$ tends to infinity as $n$ tends to infinity and for every $j\ge 1$, every $n\ge 0$, $R_{j+1,n}\le R_{j,n+1}$. Finally, we remark that
\[\lim_n \frac{p_j(e_n)}{p_{j}(e_{n+1})}=\lim_n\frac{(2^j)^n}{(2^j)^{n+1}}=\frac{1}{2^j}.\]
We can thus deduce that $0<\inf_n \frac{p_j(e_n)}{p_{j}(e_{n+1})}<\sup_n\frac{p_j(e_n)}{p_{j}(e_{n+1})}<\infty$ and the desired result follows from Corollary~\ref{cor Kot}.
\end{proof}

We remark that Corollary~\ref{cor Kot} does not rely on the structure of the norm $p_j$ but only on values of norms $p_j(e_n)$. In particular, unlike results in \cite{Golinski1}, this corollary can be applied to Köthe spaces $\Lambda^p(A)$ for every $p\ge 1$ and not only for $p=1$. In particular, we can apply Corollary~\ref{cor Kot} to power series spaces without assuming the nuclearity.

\begin{definition}
Let $\alpha$ be an increasing sequence in $[0,+\infty)$ tending to infinity and let $r\in \mathbb{R}\cup \{+\infty\}$. The power series space $\Lambda_r(\alpha)$ is the Fréchet space
\[\Lambda_r(\alpha)=\{x\in \mathbb{K}^{\mathbb{N}}: \sum_{n=1}^{\infty}|x_n|^2e^{2t\alpha_n}
<\infty\ \text{for all $t<r$}\}\]
endowed with the seminorms $p_{j}(x)=\sum_{n=1}^{\infty}|x_n|^2e^{2t_j\alpha_n}<\infty$ where $(t_j)$
is an increasing sequence tending to $r$.
\end{definition}

We remark that every increasing sequence $(t_j)$ tending to $r$ induces the same topology. 

\begin{cor}
Let $\Lambda_{\infty}(\alpha)$ be a power series space. If ${\sup_n(\alpha_{n+1}-\alpha_n)<\infty}$ then $\Lambda_{\infty}(\alpha)$ does not satisfy the Invariant Subset Property.
\end{cor}
\begin{proof}
Let $(t_j)$ be an increasing sequence of positive real numbers tending to infinity. The power series space $\Lambda_{\infty}(\alpha)$ coincide with the Köthe sequence space $\lambda^2(A)$ where $a_{j,n}=e^{t_j\alpha_n}$. In particular, we remark that we have $p_j(e_n)= e^{t_j\alpha_n}>0$ for every $n$. Moreover, we get $\frac{p_j(e_n)}{p_{j}(e_{n+1})}=e^{t_j(\alpha_n-\alpha_{n+1})}$ and thus for every $j$
\[0<e^{-t_j\sup_n(\alpha_{n+1}-\alpha_n)}\le \inf_n e^{t_j(\alpha_n-\alpha_{n+1})}\le \sup_n e^{t_j(\alpha_n-\alpha_{n+1})}\le 1.\]
On the other hand, we have $R_{j,n}=\frac{p_{j+1}(e_n)}{p_{j}(e_{n})}=e^{\alpha_n(t_{j+1}-t_j)}$ which is  non-decreasing in $n$ and non-decreasing in $j$ if we choose $(t_j)$ so that $(t_{j+1}-t_j)$ is non-decreasing. The desired result then follows from Corollary~\ref{cor Kot} by considering $t_j=j$.
\end{proof}

In the case of power series space $\Lambda_{r}(\alpha)$ with $r<\infty$, the sequence $(R_{j,n})_j$ cannot be non-decreasing. Therefore, we do not know if the space $\Lambda_{r}(\alpha)$ satisfies the assumptions of Theorem~\ref{technique} when $r$ is finite and $\sup_n(\alpha_{n+1}-\alpha_n)<\infty$. We thus let the following question:

\begin{problem}
Let $X=\Lambda_{r}(\alpha)$. If $r$ is finite and ${\sup_n(\alpha_{n+1}-\alpha_n)<\infty}$, can we deduce that $X$ does not satisfy the Invariant Subset Property?
\end{problem}

Notice that thanks to Goli\'{n}ski~\cite{Golinski1}, we know that if $r<\infty$, ${\sup_n(\alpha_{n+1}-\alpha_n)<\infty}$ and $\lim \frac{\log j}{\alpha_j}=0$ then $\Lambda_{r}(\alpha)$ does not satisfy the Invariant Subspace Property.

\subsection{Hereditary Invariant Subset Property}\label{Sec23}

Concerning the Hereditary Invariant Subset Property, we can deduce the following result from Theorem~\ref{technique}.

\begin{theorem}\label{hered}
Let $(X,(p_j)_j)$ be a Fréchet space. If for every $j\ge 1$, for every closed subspace $E$ of finite codimension, $p_j$ and $p_{j+1}$ are not equivalent on $E$ then $X$ does not satisfy the Hereditary Invariant Subset Property.
\end{theorem}
\begin{proof}
We first remark that since $p_1$ and $p_2$ are equivalent on $\ker p_2$, the space $\ker p_2$ cannot be a subspace of finite codimension. Without loss of generality we can thus assume that $\ker p_1$ is not a subspace of finite codimension.

The goal of this proof consists in constructing a basic sequence $(e_n)_{n\ge 1}$ in $X$ such that the closed infinite-dimensional subspace $M=\overline{\text{span}}\{e_n:n\ge 1\}$ satisfies the assumptions of Theorem~\ref{technique}.
We recall that if the sequence $(e_n)$ is chosen so that $p_1(e_n)>0$ and that there exist a constant $\Gamma$ such that for every $1\le j\le L_0$, every $L_1>L_0$, every $a_{j},\ldots,a_{L_1}\in \mathbb{K}$,
\[p_j\Big(\sum_{n=j}^{L_0} a_n e_n\Big)\le \Gamma p_j\Big(\sum_{n=j}^{L_1}a_n e_n\Big)\] then the sequence $(e_n)$ is a basic sequence (see proof of Lemma 2.2 in \cite{Menet1}). Such a sequence will be constructed by using the fact that for every $e_1,\dots,e_N\in X$, every $\varepsilon>0$, every continuous seminorm $p$ on $X$, there exists a closed subspace $E$ of finite codimension such that for every $x\in E$, every $y\in \text{span}\{e_1,\dots,e_N\}$,
\[p(y) \le (1+\varepsilon)p(x+y) \quad \text{\cite[Lemma 10.39]{Karl}}.\]

Assume that $e_1,\dots,e_{N}$ have been chosen.
We show that for every $\varepsilon> 0$, $C, M\ge 1$ and $K\ge J\ge 2$, it is possible to find elements $e_{N+1},\dots,e_{N+JM}$ such that 
\begin{enumerate}
\item  for every $n\in [N+1,N+JM]$, $p_1(e_n)>0$ , 
\item for every  $1\le j\le N$, every $a_{j},\ldots,a_{N+JM}\in \mathbb{K}$, 
\[p_j\Big(\sum_{k=j}^{N} a_k e_k\Big)\le 2^{\frac{1}{N^2}}p_j\Big(\sum_{k=j}^{N+JM}a_k e_k\Big),\]
\item for every $L_0,L_1\in [N,N+JM]$ with $L_0<L_1$, every $1\le j\le L_0$, every $a_{j},\ldots,a_{L_1}\in \mathbb{K}$, 
\[p_j\Big(\sum_{k=j}^{L_0} a_k e_k\Big)\le 5 p_j\Big(\sum_{k=j}^{L_1}a_k e_k\Big),\]
\end{enumerate}
and such that Conditions (1)-(3) in Theorem~\ref{technique} are satisfied with $e_{n_{j,m}}=e_{N+(j-1)M+m}$ and $\alpha_{j,m}=1$.

We first select vectors $(u_k)_{1\le k\le JM}$ such that for any $0\le k< JM$, for any $j\le N+k$, for any $a_1,\ldots,a_{N+k+1}\in \mathbb{K}$,

\[p_j\Big(\sum_{l=1}^{N} a_l e_l+\sum_{l=1}^{k} a_{N+l} u_{l}\Big)\le 2^{\frac{1}{2N^2JM}}p_j\Big(\sum_{l=1}^{N}a_l e_l+\sum_{l=1}^{k+1}a_{N+l}u_{l}\Big)\]
and such that for every $1\le k\le JM$
\begin{enumerate}
\item $p_1(u_k)>0$;
\item $p_{K}(u_k)<\frac{\varepsilon}{2}$;
\item $p_{K}(u_{k})< \frac{1}{4C} p_{1}(u_{k-1})$ with $u_0=e_N$.
\end{enumerate}
This is possible since for every $k<JM$ there exists a closed subspace $E$ of finite codimension such that for every $j\le N+k$, every $x\in E$, every $y\in \text{span}(\{e_n:n\le N\}\cup\{u_l:l\le k\})$, we have $p_j(y)\le 2^{\frac{1}{2N^2JM}} p_j(x+y)$ and since $\ker p_1$ is not of finite codimension.

We then select vectors $(v_k)_{1\le k\le JM}$ such that for any $0\le k< JM$, for any $j\le N+k$, for any $a_1,\ldots,a_{N+JM}\in \mathbb{K}$, any $b_{N+1},\dots,b_{N+k+1}\in \mathbb{K}$
\[p_j\Big(\sum_{l=1}^{N} a_l e_l+\sum_{l=1}^{JM} a_{N+l} u_{l}+\sum_{l=1}^{k}b_{N+l}v_{l}\Big)\le 2^{\frac{1}{2N^2JM}}p_j\Big(\sum_{l=1}^{N}a_l e_l+\sum_{l=1}^{JM} a_{N+l} u_{l}+\sum_{l=1}^{k+1}b_{N+l}v_{l}\Big),\]
and such for every $1\le j\le J$, every $k\in ((j-1)M,jM]$,
\begin{enumerate}
\item if $j>1$ then $p_{j-1}(v_{k})<\frac{1}{2}p_{1}(u_{JM+1-k})$
\item $p_{j}(v_k)> \frac{1}{\varepsilon}+\frac{\varepsilon}{2}$
\item $p_{j}(v_k)\ge C (\max_{k'<k}p_{K}(v_{k'})+2\max_{k\le JM}p_K(u_k))$,
\end{enumerate}
which is possible because $p_{j+1}$ and $p_j$ are not equivalent on any closed subspace of finite codimension.

We let $e_{N+k}=u_{k}+v_{JM+1-k}$ for every $1\le k\le JM$. 
We then have 
\begin{enumerate}
\item for every $1\le k\le (J-1)M$
\[p_1(e_{N+k})\ge p_1(u_k)-p_1(v_{JM+1-k})\ge (1-\frac{1}{2})p_1(u_k)>0,\]
and for every $(J-1)M<k\le JM$
\[p_1(e_{N+k})\ge p_1(v_{JM+1-k})-p_1(u_k)\ge \frac{1}{\varepsilon}+\frac{\varepsilon}{2}-\frac{\varepsilon}{2}>0,\]
\item  for every  $1\le j\le N$ every $a_{j},\ldots,a_{N+JM}\in \mathbb{K}$, 
\[p_j\Big(\sum_{k=j}^{N} a_k e_k\Big)\le 2^{\frac{1}{N^2}}p_j\Big(\sum_{k=j}^{N+JM}a_k e_k\Big),\]
since
\begin{align*}
p_j\Big(\sum_{k=j}^{N} a_k e_k\Big)&\le 2^{\frac{1}{2N^2JM}}p_j\Big(\sum_{k=j}^{N} a_k e_k+a_{N+1}u_1\Big) \\
&\le (2^{\frac{1}{2N^2JM}})^{JM}p_j\Big(\sum_{k=j}^{N} a_k e_k+\sum_{k=1}^{JM}a_{N+k}u_k\Big)\\
&\le (2^{\frac{1}{2N^2JM}})^{JM+1}p_j\Big(\sum_{k=j}^{N} a_k e_k+\sum_{k=1}^{JM}a_{N+k}u_k+a_{N+JM}v_{1}\Big)\\
&\le (2^{\frac{1}{2N^2JM}})^{2JM}p_j\Big(\sum_{k=j}^{N} a_k e_k+\sum_{k=1}^{JM}a_{N+k}u_k+\sum_{k=1}^{JM}a_{N+JM+1-k}v_{k}\Big)\\
&=2^{\frac{1}{N^2}}p_j\Big(\sum_{k=j}^{N} a_k e_k+\sum_{k=1}^{JM}a_{N+k}(u_k+v_{JM+1-k})\Big)\\
&=2^{\frac{1}{N^2}} p_j\Big(\sum_{k=j}^{N+JM} a_k e_k\Big),
\end{align*}
\item for every $L_0,L_1\in [N,N+JM]$ with $L_0<L_1$, every $1\le j\le L_0$, every $a_{j},\ldots,a_{L_1}\in \mathbb{K}$, 
\[p_j\Big(\sum_{k=j}^{L_0} a_k e_k\Big)\le  5  p_j\Big(\sum_{k=j}^{L_1}a_k e_k\Big)\]
since
\begin{align*}
p_j\Big(\sum_{k=j}^{L_0} a_k e_k\Big)&= p_j\Big(\sum_{k=j}^{N} a_k e_k+\sum_{k=\max\{j,N+1\}}^{L_0}a_{k}u_{k-N}+ \sum_{k=\max\{j,N+1\}}^{L_0}a_{k}v_{JM+1-k+N}\Big)\\
&\le p_j\Big(\sum_{k=j}^{N} a_k e_k+\sum_{k=\max\{j,N+1\}}^{L_0}a_{k}u_{k-N}\Big)+ p_j\Big(\sum_{k=\max\{j,N+1\}}^{L_0}a_{k}v_{JM+1-k+N}\Big)\\
&\le (2^{\frac{1}{2N^2JM}})^{2L_1-L_0-N} p_j\Big(\sum_{k=j}^{N} a_k e_k+\sum_{k=\max\{j,N+1\}}^{L_1}a_{k}u_{k-N}+\sum_{k=\max\{j,N+1\}}^{L_1}a_{k}v_{JM+1-k+N}\Big)\\
&\quad+p_j\Big(\sum_{k=j}^{N} a_l e_l+\sum_{k=\max\{j,N+1\}}^{L_1}a_{k}u_{k-N}+ \sum_{k=\max\{j,N+1\}}^{L_1}a_{k}v_{JM+1-k+N}\Big)\\
&\quad+ p_j\Big(\sum_{k=j}^{N} a_l e_l+\sum_{k=\max\{j,N+1\}}^{L_1}a_{k}u_{k-N}+ \sum_{k=L_0+1}^{L_1}a_{k}v_{JM+1-k+N}\Big)\\
&\le \Big((2^{\frac{1}{2N^2JM}})^{2L_1-L_0-N}+1+(2^{\frac{1}{2N^2JM}})^{L_0-N}\Big) p_j\Big(\sum_{k=j}^{L_1} a_k e_k\Big)\\
&\le \Big(1+2\ .\ 2^{\frac{1}{N^2}}\Big) p_j\Big(\sum_{k=j}^{L_1} a_k e_k\Big)\le  5 p_j\Big(\sum_{k=j}^{L_1} a_k e_k\Big)
\end{align*}
\end{enumerate}
Moreover, for every $1\le j \le J$, every $1\le m\le M$, we have
\begin{enumerate}
\item $p_{J-1}(e_{N+m})=p_{J-1}(u_{m}+v_{JM+1-m})< \frac{\varepsilon}{2}+\frac{\varepsilon}{4}\le \varepsilon$,
\item $p_1(e_{N+(J-1)M+m})=p_1(u_{(J-1)M+m}+v_{M+1-m})> \frac{1}{\varepsilon}+\frac{\varepsilon}{2} -\frac{\varepsilon}{2}\ge \frac{1}{\varepsilon}$,
\item for every $l\le M$, if $L=(j-1)M+m$ with $1\le j\le J$ and $1\le m\le M$ and if $L+l\le JM$ then for every $k< K$,
\[C p_k(e_{N+L+l})\le p_{k+1}(e_{N+L})\]
because if $JM+1-L\le (k+1)M$ then 
\begin{align*}
p_{k+1}(e_{N+L})=p_{k+1}(u_{L}+v_{JM+1-L})&\ge p_{k+1}(v_{JM+1-L})-p_{k+1}(u_L) \\
&\ge C (p_{k}(v_{JM+1-L-l})+p_k(u_{L+l})+p_{k+1}(u_L))-p_{k+1}(u_L)\\
&\ge C p_{k}(v_{JM+1-L-l}+u_{L+l})= Cp_k(e_{N+L+l}).
\end{align*}
and if $JM+1-L> (k+1)M$ then 
\begin{align*}
p_{k+1}(e_{N+L})=p_{k+1}(u_{L}+v_{JM+1-L})&\ge p_{k+1}(u_L)-p_{k+1}(v_{JM+1-L}) \\
&\ge p_{k+1}(u_L)-\frac{1}{2}p_{1}(u_L) \ge \frac{1}{2}p_{1}(u_L)\\
&\ge 2C p_K(u_{L+l})\ge C(p_{k}(u_{L+l})+p_k(v_{JM+1-L-l}))\\
&\ge C p_k(e_{N+L+l}).
\end{align*}
since $JM+1-L-l>kM$.
\end{enumerate}

We consider sequences $(J_s)_{s\ge 1}$, $(M_s)_{s\ge 1}$ and $(K_s)_{s\ge 1}$ such that for every $s\ge 1$, $K_s\ge J_s$ and for every $M\ge 1$ and $K\ge J\ge 2$, the set $\{s:J_s=J,\ M_s=M, \ K_s=K\}$ is infinite. We let $N_0=0$ and $N_t=\sum_{s\le t}J_s M_s$ for every $t\ge 1$.\\
We can then construct a sequence $(e_n)_{n\ge 1}$ such that 
\begin{enumerate}
\item  for every $n\ge 1$, $p_1(e_n)>0$ , 
\item for every $t\ge 1$, every $1\le j\le N_{t}$ every $a_{j},\ldots,a_{N_{t+1}}\in \mathbb{K}$, 
\[p_j\Big(\sum_{k=j}^{N_{t}} a_k e_k\Big)\le 2^{\frac{1}{N_t^2}}p_j\Big(\sum_{k=j}^{N_{t+1}}a_k e_k\Big),\]
\item for every $t\ge 1$, every $L_0,L_1\in [N_{t},N_{t+1}]$ with $1\le L_0<L_1$, every $1\le j\le L_0$, every $a_{j},\ldots,a_{L_1}\in \mathbb{K}$, 
\[p_j\Big(\sum_{k=j}^{L_0} a_k e_k\Big)\le 5 p_j\Big(\sum_{k=j}^{L_1}a_k e_k\Big),\]
\end{enumerate}
and such that Conditions (1)-(3) in Theorem~\ref{technique} are satisfied for $\varepsilon=\frac{1}{s}$, $C=s$, $J=J_s$, $M=M_s$ and $K=K_s$ by considering $e_{n_{j,m}}=e_{N_s+(j-1)M_s+m}$ and $\alpha_{j,m}=1$. We then deduce that the sequence $(e_n)$ satisfies Conditions (1)-(3) in Theorem~\ref{technique} for every $\varepsilon>0$, $C,M\ge 1$ and $K\ge J\ge 2$ and that $(e_n)_{n\ge 1}$ is a basic sequence since for every $n\ge 1$, $p_1(e_n)>0$ and since for every $1\le j\le L_0$, every $L_1>L_0$, every $a_{j},\ldots,a_{L_1}\in \mathbb{K}$, if $L_0,L_1\in [N_{t},N_{t+1}]$ we have $p_j\Big(\sum_{k=j}^{L_0} a_k e_k\Big)\le 5 p_j\Big(\sum_{k=j}^{L_1}a_k e_k\Big)$ and if $L_0\in (N_{s},N_{s+1}]$ and $L_1\in (N_{t},N_{t+1}]$ with $s<t$ then 
\begin{align*}
p_j\Big(\sum_{k=j}^{L_0} a_k e_k\Big)&\le 5p_j\Big(\sum_{k=j}^{N_{s+1}} a_k e_k\Big)\\
&\le 5 \Big(\prod_{r=s+1}^{t-1}2^{\frac{1}{N^2_r}}\Big)p_j\Big(\sum_{k=j}^{N_{t}} a_k e_k\Big)\\
&\le 5^2 \Big(\prod_{r=s+1}^{t-1}2^{\frac{1}{N^2_r}}\Big)p_j\Big(\sum_{k=j}^{L_1} a_k e_k\Big)\\
&\le 5^2 2^{\sum_{n=1}^{\infty}\frac{1}{n^2}} p_j\Big(\sum_{k=j}^{L_1}a_k e_k\Big).
\end{align*}
We conclude that $(e_n)$ is a basic sequence and that the space $M=\overline{\text{span}}\{e_n:n\ge 1\}$ satisfies the assumptions of Theorem~\ref{technique}. Therefore the subspace $M$ does not satisfy the Invariant Subset Property and thus $X$ does not satisfy the Hereditary Invariant Subspace Property. This concludes the proof.
\end{proof}

\begin{cor}\label{heredker}
Let $(X,(p_j)_j)$ be a Fréchet space. If for every $j\ge 1$, $\ker p_{j+1}$ is a subspace of infinite codimension in $\ker p_j$ then $X$ does not satisfy the Hereditary Invariant Subset Property.
\end{cor}
\begin{proof}
Let  $j\ge 1$ and $E$ a closed subspace of finite codimension. Since $\ker p_{j+1}$ is a subspace of infinite codimension in $\ker p_j$, we deduce that $E\cap \ker p_{j}$ cannot be included in $\ker p_{j+1}$. We can thus find $x\in E$ such that $p_j(x)=0$ and $p_{j+1}(x)>0$. The seminorms $p_j$ and $p_{j+1}$ are thus not equivalent on $E$ and we can conclude thanks to Theorem~\ref{hered}.
\end{proof}

If we now consider a non-normable Fréchet space $X$ with a continuous norm, there then exists a sequence of norms $(p_j)_{j\ge 1}$ inducing the topology of $X$ such that for every $j$, $p_{j}$ and $p_{j+1}$ are not equivalent on $X$. However, it is possible that there exists a closed subspace  $E$ of finite codimension such that $p_j$ and $p_{j+1}$ are equivalent on $E$. Indeed, if we consider, for instance, $H(\mathbb{C})$, which can be identified with $\lambda^1(A)$ where $a_{j,k}=j^k$ and if we let $q_j(x)=p_j(x)+|x^*_j(x)|$ where $x^*_j(x)=\sum_{n=0}^{\infty}(j+1)^nx_n$, then the sequence of norms $(p_1,q_1,p_2,q_2,p_3,q_3,\dots)$ induces the same topology that $(p_j)_{j\ge 1}$ (since $p_j(x)\le q_j(x)\le 2p_{j+1}(x)$). However, while $p_j$ and $q_j$ are not equivalent on $X$, these norms are equivalent on $\ker x^*_j$. 

One can therefore wonder if for every non-normable Fréchet space $X$ with a continuous norm, there exists a sequence of norms $(p_j)$ inducing the topology of $X$ and such that $p_{j+1}$ and $p_j$ are not equivalent on every closed subspace of finite codimension.

\begin{problem}
Does every non-normable Fréchet space $X$ with a continuous norm not satisfy the Hereditary Invariant Subset Property?
\end{problem}

We can show that the answer is yes in the case of Köthe sequence spaces.
\begin{cor}
Let $X=\lambda^p(A)$ or $c_0(A)$ be a Köthe sequence space with a continuous norm. If $X$ is non-normable then $X$ does not satisfy the Hereditary Invariant Subset Property.
\end{cor}
\begin{proof}
Since $X$ possesses a continuous norm, there exists $j$ such that $a_{j,n}\ne 0$ for every $n$ (Theorem~\ref{Schauder2}). Therefore, since $X$ is not a Banach space, there exists an increasing sequence $(j_m)$ such that for every $C>0$, every $m\ge 1$, the set $\{n: a_{j_{m+1},n}> C a_{j_{m},n}\}$ is infinite. Without loss of generality, we can thus assume that $p_1$ is a norm and that for every $C>0$, every $j\ge 1$, the set $I_{j,C}=\{n: a_{j+1,n}> C a_{j,n}\}$ is infinite. Let $j\ge 1$ and $E$ be a closed subspace of finite codimension. For every $C>0$, we can then find a non-zero vector $x\in \text{span}\{e_{n}:n\in I_{j,C}\}\cap E$. This implies that $p_{j+1}$ and $p_j$ are not equivalent on $E$ since $p_j(x)\ne 0$ and
\[p_{j+1}^p(x)=\sum_{n\in I_{j,C}}|x_na_{j+1,n}|^p\ge C^p \sum_{n\in I_{j,C}}|x_n a_{j,n}|^p=C^p p_{j}^p(x)\]
if $X=\lambda^p(A)$ and
\[p_{j+1}(x)=\sup_{n\in I_{j,C}}|x_na_{j+1,n}|\ge C \sup_{n\in I_{j,C}}|x_n a_{j,n}|=C p_{j}(x)\]
if $X=c_0(A)$. The desired result then follows from Theorem~\ref{hered}.
\end{proof}

\section{Proof of Theorem~\ref{technique}}\label{last}
Let $(X,(p_j)_{j})$ be a Fréchet space possessing a Schauder basis $(e_n)_{n\ge 0}$ satisfying the assumptions of Theorem~\ref{technique}. These assumptions are still satisfied if we consider a subsequence of $(p_j)$ and if we consider multiples of these seminorms. Without loss of generality, we can thus assume that $p_{j+1}(x)\ge 2 p_j(x)$ for every $j\ge 1$ and that there exists an increasing sequence $(C_j)_{j\ge 1}$ such that for every $j\ge 1$, every $M\le N$, every $x_1,\dots,x_N$, 
\begin{equation}\label{basis}
p_j\Big(\sum_{n=0}^{M}x_n e_n\Big)\le C_j p_{j+1}(\sum_{n=0}^{N}x_n e_n)\quad \text{(see Theorem~\ref{Schauder}).}
\end{equation}
Moreover, in view of Theorem~\ref{Schauder2}, we can also assume that $p_1$ is a norm.

Let $(N_n)_{n\ge 1}$ and $(R_n)_{n\ge 1}$ with $\max\{N_n,R_n\}\le n$ such that $|N_{n+1}-N_n|\le 1$ for every $n$ and such that for every $(i,j)\in \mathbb{N}^2$, there exists an infinity of integers $n$ such that $(N_n,R_n)=(i,j)$.

Let $(\mu_n)$ be a sequence of positive integers and sequences $(\Delta_n)_{n\ge 1}$, $(a_n)_{n\ge 1}$, $(t_n)_{n\ge 0}$ and $(c_n)_{n\ge 1}$ such that $\Delta_1=1$,  $\Delta_{n+1}=c_{n}^{\mu_n}+t_n$ for every $n\ge 1$ and such that for every $n\ge 1$, every $1\le k\le \mu_{n}-1$,
\begin{equation}\label{param}
2\Delta_n<a_n,\quad a_{n}+\Delta_{n}<t_{n},\quad 3t_{n}+n< c_{n} \quad \text{and}\quad 2c_n^k+2t_n<c_n^{k+1}.
\end{equation}
Notice that we then have for every $n\ge 1$, 
\[a_{n+1}+\Delta_{n+1}+c_n^{\mu_n}<c_{n+1}.\]

Let $(Q_{n,k})_{n\ge 1, 1\le k\le \mu_{n}}$ be a family of polynomials satisfying $\text{deg}\, Q_{n,k}\le t_n+n$.

We want to construct an operator without non-trivial invariant subset, i.e. an operator such that every non-zero vector is hypercyclic. This will be done by using Read's type constructions (see \cite{Read3} or \cite{Grivaux} for a generalization). Actually we will look at sufficient conditions to apply a Read's type construction and we will show at the end of the proof that these conditions can be satisfied under the assumptions of Theorem~\ref{technique}. To this end, we will work with multiples of the basis $(e_n)_{n\ge 0}$ which we will reorder in a suitable way. 

We denote by $(u_n)_{n\ge 0}$ the sequence obtained by this reordering and we consider the map $T:\text{span}\{u_j:j\ge 0\}\to \text{span}\{u_j:j\ge 0\}$ defined by

\[T^{j}u_0=\left\lbrace
\begin{array}{ll}
u_j+T^{j-a_n}u_0 & \mbox{if $j\in [a_n,a_n+\Delta_n)$;}\\
u_j+Q_{n,k}(T)T^{j-c^k_n}u_0 & \mbox{if $j\in [c^k_n,c^k_n+t_n)$,  $1\le k\le \mu_n$;}\\
u_j & \mbox{otherwise.}
\end{array}
\right.\]

The goal of this proof consists in showing that for a convenient reordering $(u_n)$ of multiples of the basis $(e_n)$ and for a good choice of parameters, the operator $T$ can be extended on $X$ and does not possess non-trivial invariant subsets. We will thus assume that $e_n=\lambda_{\sigma(n)} u_{\sigma(n)}$ where $\sigma$ is a bijection on $\mathbb{Z}_+$ and $\lambda_m\ne 0$ for every $m$. Moreover, the sequence of parameters $(t_n)$ will be chosen such that for every $n\ge 1$,
\begin{equation}
\{\lambda_{0} u_0,\dots,\lambda_{t_n-1} u_{t_n-1}\}=\{e_0,\dots,e_{t_n-1}\}.
\label{tn}
\end{equation}

We first remark that we can easily compute the vectors $Tu_j$ for every $j\ge 0$ from the iterates of $u_0$ under the action of $T$.

\begin{prop}\label{comp}\text{~}
\begin{enumerate}
\item For every $j\in  \mathbb{Z}_+\backslash\bigcup_{n\ge 1}\bigcup_{1\le k\le \mu_n}\{a_n-1,a_n+ \Delta_n-1,c_n^k-1,c_n^k+t_n-1\}$, \[Tu_j=u_{j+1};\]
\item For every $n\ge 1$,
\[T u_{a_n-1}=u_{a_n}+u_0;\]
\item For every $n\ge 1$,
\[T u_{a_n+\Delta_n-1}=u_{a_n+\Delta_n}-u_{\Delta_n}\]
\item For every $n\ge 1$, every $1\le k\le \mu_n$,
\[T u_{c_n^k-1}=u_{c_n^k}+Q_{n,k}(T)u_0;\]
\item For every $n\ge 1$, every $1\le k\le \mu_n$,
\[T u_{c_n^k+t_n-1}= u_{c_n^k+t_n}-Q_{n,k}(T)u_{t_n}.\]
\end{enumerate}
\end{prop}

We can then show that under some conditions, the operator $T$ can be extended continuously on $X$. We recall that if $P(T)=\sum_{n=0}^d \lambda_n T^n$ then we let $|P|=\sum_{n=0}^{d}|\lambda_n|$.

\begin{lemma}[Continuity]\label{cont}
If for every $n\ge 1$,
\begin{enumerate}
\item for every $1\le l\le n$, every $j\in [t_n,t_{n+1})$, we have
\begin{equation}
p_{l+1}(u_{j})\ge 2^{j+1} p_{l}(u_{j+1}),
\label{finalcont1}
\end{equation}
\item we have
\begin{equation}
p_{1}(u_{a_n-1})\ge 2^{a_n}p_{n-1}(u_0)
\label{finalcont2}
\end{equation}
\item we have
\begin{equation}
p_{1}(u_{a_n+\Delta_n-1})\ge 2^{a_n+\Delta_n}p_{n-1}(u_{\Delta_n}).
\label{finalcont3}
\end{equation}
\item for every $1\le k\le \mu_n$, we have
\begin{equation}
p_{1}(u_{c^k_n-1})\ge 2^{c_n^k}|Q_{n,k}|\sup_{j< t_n}p_n(T^ju_0)
\label{finalcont4}
\end{equation}
\item for every $0\le m\le \text{\emph{deg}}(Q_{n,k})$, we have
\begin{equation}
p_{n}(u_{t_n+m})\le \frac{1}{2^{c_n^k+t_n}|Q_{n,k}|}\min\{p_{1}(u_{c_n^k+t_n-1}),p_{1}(u_{c_n^k-1})\}
\label{finalcont5}
\end{equation}
\end{enumerate}
then for every $N\ge 1$, every $j\ge t_N$, we have $p_{N}(Tu_j) \le \frac{1}{2^j}p_{N+1}(u_j)$.
\end{lemma}
\begin{proof}
Let $N\ge 1$ and $j\in [t_n,t_{n+1})$ with $n\ge N$. Thanks to Proposition~\ref{comp}, we get the following inequalities:
\begin{enumerate}
\item If $j\in  \mathbb{Z}_+\backslash\bigcup_{n\ge 1}\bigcup_{1\le k\le \mu_n}\{a_n-1,a_n+ \Delta_n-1,c_n^k-1,c_n^k+t_n-1\}$ then it follows from \eqref{finalcont1} that
\[p_{N}(Tu_j)=p_N(u_{j+1}) \le \frac{1}{2^j}p_{N+1}(u_j).\]
\item If $j=a_{n+1}-1$ then it follows from \eqref{finalcont1} and \eqref{finalcont2} that
\[p_N(T u_j)\le p_N(u_{a_{n+1}})+p_N(u_0)\le \frac{1}{2^{a_{n+1}-1}}p_{N+1}(u_{a_{n+1}-1}).\]
\item If $j=a_{n+1}+\Delta_{n+1}-1$ then it follows from \eqref{finalcont1} and \eqref{finalcont3} that
\[p_N(T u_{j})\le p_N(u_{a_{n+1}+\Delta_{n+1}})+p_N(u_{\Delta_{n+1}})\le \frac{1}{2^{a_{n+1}+\Delta_{n+1}-1}}p_{N+1}(u_{a_{n+1}+\Delta_{n+1}-1}).\]
\item If $j= c_n^k-1$ for some $1\le k\le \mu_n$ then since $\deg(Q_{n,k})\le t_n+n<c_n$, it follows from \eqref{finalcont1}, \eqref{finalcont4} and \eqref{finalcont5} that
\begin{align*}
p_{N}(Tu_j)&\le p_N(u_{c_n^k})+|Q_{n,k}|\sup_{m\le \text{deg}\, Q_{n,k}}p_{N}(T^mu_0)\\
&\le \frac{1}{2^{c^k_n}} p_{N+1}(u_{c_n^k-1})+|Q_{n,k}|\max\Big\{\sup_{m<t_n}p_{n}(T^mu_0),\sup_{0\le m\le \deg(Q_{n_k})-t_n}p_{n}(T^{t_n+m}u_0)\Big\}\\
&=  \frac{1}{2^{c^k_n}} p_{N+1}(u_{c_n^k-1})+|Q_{n,k}|\max\Big\{\sup_{m<t_n}p_{n}(T^mu_0),\sup_{0\le m\le \deg(Q_{n_k})-t_n}p_{n}(u_{t_n+m})\Big\}\\
&\le \frac{1}{2^{c^k_n-1}}p_{N+1}(u_{c^k_n-1}).
\end{align*}
\item If $j=c_n^k+t_n-1$ then since $t_n+\deg(Q_{n,k})\le 2t_n+n<c_n$, it follows from \eqref{finalcont1} and \eqref{finalcont5} that
\begin{align*}
p_N(T u_{j})&\le p_N(u_{c_n^k+t_n})+|Q_{n,k}|\sup_{m\le \text{deg}\, Q_{n,k}}p_N(T^mu_{t_n})\\
&\le p_N(u_{c_n^k+t_n})+|Q_{n,k}|\sup_{m\le \text{deg}\, Q_{n,k}}p_n(u_{t_n+m})\\
&\le \frac{1}{2^{c_n^k+t_n-1}}p_{N+1}(u_{c_n^k+t_n-1}).
\end{align*}
\end{enumerate}
\end{proof}

\begin{cor}\label{contnorm}
If \eqref{finalcont1}-\eqref{finalcont5} are satisified then the map $T$ defined on $\text{\emph{span}}\{u_j:j\ge 0\}$ can be uniquely extended to a continuous operator $T$ on $X$ satisfying for every $N\ge 1$, every $x\in X$, 
\[p_N(Tx)\le 4C_{N+1}L_Np_{N+2}(x),\] where $C_{N+1}$ is given by \eqref{basis} and $L_N$ only depends on $\{u_0,\dots,u_{t_N-1}\}$.
\end{cor}
\begin{proof}
Let $N\ge 1$ and $x\in X$. By \eqref{basis}, there exists a constant $C_{N+1}$ such that for every $M\ge 0$, every $y=\sum_{j=0}^{\infty}y_je_j\in X$,
\[p_{N+1}\Big(\sum_{j=0}^{M}y_je_j\Big)\le C_{N+1} p_{N+2}(y).\]
 Let $\sigma$ be a bijection from $\mathbb{Z}_+$ to $\mathbb{Z}_+$ and $(\lambda_n)$ a sequence of non-zero scalars such that $e_j=\lambda_{\sigma(j)}u_{\sigma(j)}$.
Then there exists $L_N\ge 1$ depending on $\{u_0,\dots,u_{t_N-1}\}$ such that  
\begin{align*}
p_N(Tx)&=p_N\Big(T\big(\sum_{j=0}^{\infty}x_j\lambda_{\sigma(j)}u_{\sigma(j)}\big)\Big)\\
&\le \sum_{0\le \sigma(j)<t_N}p_{N}(T(x_j\lambda_{\sigma(j)} u_{\sigma(j)}))+ \sum_{\sigma(j)\ge t_N}\frac{1}{2^{\sigma(j)}}p_{N+1}(x_j\lambda_{\sigma(j)} u_{\sigma(j)})\\
&\le\sum_{0\le \sigma(j)<t_N}\frac{L_N}{2^{\sigma(j)}}p_{N+1}(x_j\lambda_{\sigma(j)} u_{\sigma(j)})+  \sum_{\sigma(j)\ge t_N}\frac{1}{2^{\sigma(j)}}p_{N+1}(x_j\lambda_{\sigma(j)} u_{\sigma(j)})\\
&\le L_N\sum_{j=0}^{\infty}\frac{1}{2^{\sigma(j)}}p_{N+1}(x_j\lambda_{\sigma(j)} u_{\sigma(j)})=L_N\sum_{j=0}^{\infty}\frac{1}{2^{\sigma(j)}}p_{N+1}(x_j e_j)\\
&\le L_N\sum_{j=0}^{\infty}\frac{1}{2^{\sigma(j)}}\Big(p_{N+1}\big(\sum_{i=0}^jx_i e_i\big)+p_{N+1}\big(\sum_{i=0}^{j-1}x_i e_i\big)\Big)\\
&\le L_N\sum_{j=0}^{\infty}\frac{1}{2^{\sigma(j)}}2C_{N+1}p_{N+2}(x)=4C_{N+1} L_N p_{N+2}(x),
\end{align*}
where the first inequality follows from Lemma~\ref{cont}.
\end{proof}
\begin{remark}
In the construction of Goli\'{n}ski, it is only required that for every $N\ge 1$, there exists a constant $C>0$ such that $p_N(Tu_j)\le Cp_N(u_j)$. The continuity of $T$ on $X$ can then only be obtained if $\sum_{j=0}^{\infty}p_N(x_ju_j)\le K p_N\Big(\sum_{j=0}^{\infty} x_n u_n\Big)$, i.e. if $p_N$ is a norm of type $\ell_1$ for the sequence $(u_j)_{j\ge 0}$.
\end{remark}

We are now going to investigate under which conditions we can assert that every non-zero vector of $X$ is hypercyclic, i.e. $\{T^nxn\ge 0\}$ is dense in $X$. We start by fixing the family of polynomials $Q_{n,k}$. The choice of this family will follow from the following lemma (see for instance \cite[Lemma 3]{Golinski2}).

\begin{lemma}\label{P}
Let $\varepsilon>0$, let $a$ and $t$ be positive integers with $t>a$ and $(\gamma_0,\dots,\gamma_{t-1})$ be a perturbed canonical basis of $\text{\emph{span}}(u_0,\dots, u_{t-1})$ satisfying 
\[\gamma_0=u_0 \quad \text{and} \quad \gamma_a=\varepsilon u_a+u_0.\]
Let $\|\cdot\|$ be a norm on $\text{span}\{u_j:j\ge 0\}$ and $K\subset \text{\emph{span}}(u_0,\dots, u_{t-1})$ be a compact set in the induced topology such that $\nu:=a-\text{val}_{\gamma}(K)\ge 0$.\\
Then there is a number $D\ge 1$ satisfying
\[\sum_{j=0}^{t-1}|\lambda_j|\le D \quad\text{for every $y=\sum_{j=0}^{t-1}\lambda_j \gamma_j\in K$}\]
and a finite family of polynomials $(P_l)_{l=1}^L$ satisfying for every $1\le l\le L$,
\[ \deg P_l< t \quad\text{and}\quad |P_l|\le D\]
such that for any $y \in K$ there is $1\le l\le L$ such that for each perturbed forward
shift $T:\text{\emph{span}}\{u_j:j\ge 0\}\to \text{\emph{span}}\{u_j:j\ge 0\}$ satisfying $T^ju_0=\gamma_j$ for every $1\le j\le t-1$, we have
\[\|P_l(T)y-u_0\|\le 2\varepsilon \|u_a\|+ D\max_{t\le j\le 2t}\|T^ju_0\|.\]
\end{lemma}
\begin{remark}
The existence of a constant $D\ge 1$ such that $\sum_{j=0}^{t-1}|\lambda_j|\le D$ for every $y=\sum_{j=0}^{t-1}\lambda_j \gamma_j\in K$ and such that $|P_l|\le D$ for every $1\le l\le L$ follows directly from the compacity of $K$ and from the fact that the family $(P_l)_{l=1}^L$ is finite. 
\end{remark}

We will consider the sequence of compact sets $(K_n)_{n\ge 2}$ given by 
\begin{equation}
K_n=\Big\{y\in \text{span}(u_0,\dots,u_{t_{n}-1})~:~ p_{1}(y)\le \frac{3}{2}\ \text{and}\  p_{1}(\tau_n y)\ge 1/2\Big\}
\end{equation}
where for every $n\ge 2$
\begin{equation}
\tau_n\Big(\sum_{j =0}^{t_{n}-1}y_jT^ju_0\Big)=\sum_{j=0}^{t_{n-1}-1}y_jT^ju_0.
\end{equation}
\begin{cor}\label{cor Pn}
For every $n\ge 2$, there exist $D_n\ge 1$ satisfying
\[\sum_{j=0}^{t_n-1}|y_j|\le D_n \quad\text{for every $y=\sum_{j=0}^{t_n-1}y_j T^ju_0\in K_n$}\]
and a family of polynomials $\mathcal{P}_n=(P_{n,k})_{k=1}^{k_n}$ satisfying 
\[\deg(P_{n,k})<t_n \quad\text{and}\quad |P_{n,k}|\le D_n\]
such that for each $y\in K_n$, there exists $1\le k\le k_n$ such that
\[p_{N_n}(P_{n,k}(T)y-u_0)\le 2p_{N_n}(u_{a_n})+ D_n\max_{t_n\le j\le 2t_n}p_{N_n}(u_j).\]
Moreover, $D_n$ and $\mathcal{P}_n$ do not depend on $u_j$ for $j\ge t_n$.
\end{cor}
\begin{proof}
It suffices to apply Lemma~\ref{P} to $K_n$ and the perturbed canonical basis $(u_0,Tu_0,\dots,T^{t_n-1}u_0)$ by remarking that $\nu\ge a_n-t_{n-1}\ge 0$, that $\varepsilon=1$ since $T^{a_n}u_0=u_{a_n}+u_0$ and that $T^j u_0=u_j$ for every $j\in [t_n, 2t_n]$.
\end{proof}

If we assume that $u_0,\cdots, u_{t_n-1}$ are constructed, we can then consider the family $\mathcal{P}_n=(P_{n,k})_{k=1}^{k_n}$ given by Corollary~\ref{cor Pn} since this family does not depend on the elements $u_j$ with $j\ge t_n$. Moreover, we can consider a net $\mathcal{S}_n=(S_{n,w})_{w=1}^{W_n}$ of polynomials such that $\text{deg}(S_{n,w})\le R_n$, $|S_{n,w}|\le R_n$ and such that for every polynomial $S$ satisfying $\text{deg}(S)\le R_n$ and $|S|\le R_n$, there exists $1\le w\le W_n$ satisfying 
\[
|S-S_{n,w}|\le \frac{1}{\max_{l\le R_n}p_{N_n}(T^lu_0)}.\]
We then consider for $(Q_{n,l})_{l=1}^{\mu_n}$ the enumeration of polynomials $P_{n,k}S_{n,w}$. In particular, we get 
\begin{equation}
\text{deg}\, Q_{n,l}\le t_n+n \quad \text{and}\quad |Q_{n,l}|\le nD_n
\end{equation}
since $R_n\le n$.

The following lemma shows that under some conditions, we can truncate any non-zero vector $x$ such that the truncated part belongs to $K_n$ for some $n$.

\begin{lemma}[Sets $K_n$]\label{Kn}
If for every $n\ge 1$, every $j\in \bigcup_{l=1}^{\mu_{n}}[c_{n}^l,c_{n}^l+t_{n})\cup [a_{n+1},a_{n+1}+\Delta_{n+1})$
\begin{equation}
\label{K1}
p_{N_{n+1}+2}(u_j)\ge nD_{n} 2^j\sup_{l< t_{n}}p_{1}(T^{l}u_0),
\end{equation}
then for every $N$, every sequence $(n_k)$ such that $N_{n_k}=N$ and every $x\in X$ with $p_1(x)=1$, we have that for all but finitely many $k$
\[\pi_{t_{n_k}}x\in K_{n_k}\]
where $\pi_{t_{n_k}}x=\sum_{l=0}^{t_{n_k-1}}x_l e_l$. 
\end{lemma}
\begin{proof}
We recall that \begin{equation*}
K_n=\Big\{y\in \text{span}(u_0,\dots,u_{t_{n}-1})~:~ p_{1}(y)\le \frac{3}{2}\ \text{and}\  p_{1}(\tau_n y)\ge 1/2\Big\}
\end{equation*}
and that
\begin{equation*}
\tau_n\Big(\sum_{m =0}^{t_{n}-1}y_mT^mu_0\Big)=\sum_{m=0}^{t_{n-1}-1}y_mT^mu_0.
\end{equation*}
Let $n\ge 2$. We start by computing $p_{1}(\tau_n u_j)$ for every $j\in [0,t_n)$.
\begin{itemize}
\item If $j< t_{n-1}$ then $\tau_n u_j=u_j$.
\item If $j\ge t_{n-1}$ and $j\notin \bigcup_{l=1}^{\mu_{n-1}}[c_{n-1}^l,c_{n-1}^l+t_{n-1})\cup [a_n,a_n+\Delta_n)$ then $\tau_n u_j=0$. 
\item If $j\in [c_{n-1}^l,c_{n-1}^l+t_{n-1})$ then $\tau_n u_j=-\tau_nQ_{n-1,l}T^{j-c_{n-1}^l}u_0$ and thus by~\eqref{K1}
\[p_{1}(\tau_n u_j)\le |Q_{n-1,l}|\max_{0\le m< t_{n-1}}p_{1}(T^{m}u_0)\le \frac{1}{2^j} p_{N_{n}+2}(u_j)\]
since $|Q_{n-1,l}|\le (n-1)D_{n-1}$.
\item If $j\in [a_n,a_n+t_{n-1})$ then $\tau_n u_j=-T^{j-a_n}u_0$ and thus by \eqref{K1}
\[p_{1}(\tau_n u_j)=p_{1}(T^{j-a_n}u_0)\le \frac{1}{2^j} p_{N_{n}+2}(u_j).\]
\item If $j\in [a_n+t_{n-1},a_n+\Delta_n)$ then $\tau_n u_j=-\tau_nT^{j-a_n}u_0=0$.
\end{itemize}
Let $y\in \text{span}\{u_j:t_{n-1}\le j< t_{n}\}$. We deduce that 
\begin{align*}
p_{1}(\tau_n y)&\le \sum_{j=c_{n-1}}^{a_n+\Delta_n-1}p_{1}(y_j \tau_n u_j)\le \sum_{j=c_{n-1}}^{a_n+\Delta_n-1}\frac{1}{2^j}p_{N_{n}+2}(y_j u_j)\\
&\le \sum_{j=c_{n-1}}^{a_n+\Delta_n-1}\frac{2C_{N_{n}+2}}{2^j}p_{N_{n}+3}(y)\le \frac{2C_{N_{n}+2}}{2^{c_{n-1}-1}}p_{N_{n}+3}(y).
\end{align*}
where $C_{N_{n}+2}$ is given by \eqref{basis}.\\
Let $N\ge 1$, let $(n_k)$ be a sequence such that $N_{n_k}=N$ and let $x=\sum_{i=0}^{\infty}x_i e_i\in X$ with $p_1(x)=1$. It follows from \eqref{tn} that for every $n\ge 2$, $\text{span}\{u_{t_{n-1}},\dots,u_{t_n-1}\}=\text{span}\{e_{t_{n-1}},\dots,e_{t_n-1}\}$ and therefore, if we let $\pi_{[I,J)}x=\sum_{i=I}^{J-1}x_i e_i$, we get
\begin{align*}
p_1(\tau_{n_k}\pi_{t_{n_k}}x)&=p_1(\tau_{n_k}\pi_{[0,t_{n_k-1})}x+\tau_{n_k}\pi_{[t_{n_k-1},t_{n_k})}x)\\
&\ge p_1(\pi_{t_{n_k-1}}x)-p_1(\tau_{n_k}\pi_{[t_{n_k-1},t_{n_k})}x)\\
& \ge p_1(\pi_{t_{n_k-1}}x)-\frac{2C_{N+2}}{2^{c_{n_k-1}-1}}p_{N+3}(\pi_{[t_{n_k-1},t_{n_k})}x)\xrightarrow[k\to \infty]{} 1.
\end{align*}
We conclude that for all but finitely many $k$, $\pi_{t_{n_k}}x\in K_{n_k}$.
\end{proof}

Let $x\in X$ such that $p_1(x)=1$. Thanks to Corollary~\ref{cor Pn} and Lemma~\ref{Kn}, we can obtain interesting approximations with the elements $Q_{n,k}(T)(\pi_{t_n}x)$. However, in order to prove that $x$ is a hypercyclic vector, we still need to be able to control the tail $x-\pi_{t_n}x$.

\begin{lemma}[Tails]\label{tail}
Let $(C_{j})_{j\ge 1}$ be the basis constants given by \eqref{basis}.
If for every $n\ge 1$, every $j\in [t_n,t_{n+1})$
\begin{enumerate}
\item for every $l\le n$, every $1\le r\le c_{n-1}^{\mu_{n-1}}$,
\begin{equation}
\label{tail1}
p_{l+2}(u_{j})\ge 2^{j+2}C_{n+2} p_l(u_{j+r});
\end{equation}
\item for every $1\le k\le \mu_n$,
\begin{equation}
\label{tail2}
p_{N_n+2}(u_{j})\ge 2^{j+2} C_{n+2}p_{N_n}(u_{j+c_n^k});
\end{equation} 
\item if $j\in [a_{n+1}+\Delta_{n+1}-c_n^{\mu_n},a_{n+1}+\Delta_{n+1})$ then 
\begin{equation}
\label{tail3}
p_{1}(u_j)\ge 2^{a_{n+1}+\Delta_{n+1}+1}C_{n+2}\sup_{m\in [\Delta_{n+1},\Delta_{n+1}+c_n^{\mu_n})}p_{n}(u_m);
\end{equation} 
\item if $j\in [a_{n+1}-c_n^{\mu_n},a_{n+1})$ then
\begin{equation}
\label{tail4}
p_1(u_j)\ge 2^{a_{n+1}+1}C_{n+2}\sup_{m\le c_{n}^{\mu_n}}p_n(T^m u_0);
\end{equation} 
\item if $j\in [c_n^{k},c_n^{k}+t_n)$ with $1\le k\le \mu_n$ then
\begin{equation}
\label{tail5}
p_1(u_j)\ge 2^{c_n^{k}+t_n+2}C_{n+2} D_n\sup_{m\in [t_n,3t_n+n)}p_{n}(u_m)
\end{equation} 
and
\begin{equation}
\label{tail6}
p_{N_{n}+2}(u_j)\ge 2^{c_n^{k}+t_n+2}C_{n+2}D_n \sup_{m\in \bigcup_{1\le k'\le \mu_n}[c_n^{k'},c_{n}^{k'}+2t_n+n]}p_{N_n}(u_m);
\end{equation} 
\item if $j\in [c_n^{k}-c_{n'}^{k'},c_n^{k})$ with $(n',k')<(n,k)$ then
\begin{equation}
\label{tail7}
p_{1}(u_j)\ge 2^{c_n^k+1}C_{n+2}D_n\sup_{m\le 2t_n+n}p_{n}(T^mu_0);
\end{equation} 
\end{enumerate}
then for every $n\ge 1$, every $1\le k \le \mu_n$, every $x\in \overline{\text{\emph{span}}}\{e_j:j\ge t_n\}$, we have
\[p_{N_n}(T^{c_n^k}x)\le p_{N_n+3}(x).\]  
\end{lemma}
\begin{proof}
Let $n\ge 1$, $1\le k \le \mu_n$ and $x\in \overline{\text{span}}\{e_j:j\ge t_n\}$. We first remark that if for every $j\ge t_n$, 
\[ p_{N_n}(T^{c_n^k}u_j)\le \frac{1}{2^{j+1}C_{n+2}}p_{N_{n}+2}(u_j)\]
then we have
\[p_{N_n}(T^{c_n^k}x)\le p_{N_n+3}(x).\]  
Indeed, since $\text{span}\{u_j:j\ge t_n\}=\text{span}\{e_j:j\ge t_n\}$ by \eqref{tn}, if we write $e_j=\lambda_{\sigma(j)}u_{\sigma(j)}$ then
\begin{align*}
p_{N_n}(T^{c_n^k}x)&\le  \sum_{j=t_n}^{\infty}|x_j|p_{N_n}(T^{c_n^k}e_j)\le \sum_{j=t_n}^{\infty}|x_j|\frac{1}{2^{\sigma(j)+1}C_{n+2}}p_{N_{n}+2}(e_j)\\
&\le \sum_{j=t_n}^{\infty}\frac{C_{N_n+2}}{2^{\sigma(j)}C_{n+2}}p_{N_{n}+3}(x)\le p_{N_{n}+3}(x) \quad\text{since $(C_j)_j$ is an increasing sequence}.
\end{align*}
Let $j\in [t_n,t_{n+1})$, $n'\le n$ and $1\le k'\le \mu_{n'}$. It is thus sufficient to show that if Conditions \eqref{tail1}-\eqref{tail7} are satisfied then we have
\begin{equation}\label{but}
 p_{N_{n'}}(T^{c_{n'}^{k'}}u_j)\le \frac{1}{2^{j+1}C_{n+2}}p_{N_{n'}+2}(u_j).
 \end{equation}
 We discuss the different possibilities for $j$.
\begin{itemize}
\item If $j\in[a_{n+1}+\Delta_{n+1},t_{n+1})$ then 
\[T^{c_{n'}^{k'}}u_j=u_{j+c_{n'}^{k'}}\]
since $t_{n+1}+c_n^{\mu_n}<c_{n+1}$ and \eqref{but} follows from \eqref{tail1} for $n'<n$ and from \eqref{tail2} for $n'=n$.
\item If $j\in [a_{n+1}+\Delta_{n+1}-c_{n'}^{k'} , a_{n+1}+\Delta_{n+1})$ then
\[T^{c_{n'}^{k'}}u_j=T^{c_{n'}^{k'}}T^ju_0-T^{c_{n'}^{k'}}T^{j-a_{n+1}}u_0=u_{j+c_{n'}^{k'}}-u_{j-a_{n+1}+c_{n'}^{k'}}\]
since $j+c_{n'}^{k'}\in [a_{n+1}+\Delta_{n+1},c_{n+1})$ and since $j-a_{n+1}+c_{n'}^{k'}\in [\Delta_{n+1},\Delta_{n+1}+c_{n'}^{k'})$ and $\Delta_{n+1}+c_n^{\mu_n}<a_{n+1}$.
It follows from \eqref{tail1}, \eqref{tail2} and \eqref{tail3}  that
\begin{align*}
 p_{N_{n'}}(T^{c_{n'}^{k'}}u_j)&\le \frac{1}{2^{j+2}C_{n+2}}p_{N_{n'}+2}(u_j)+\frac{1}{2^{a_{n+1}+\Delta_{n+1}+1}C_{n+2}}p_{1}(u_j)\\
&\le \frac{1}{2^{j+1}C_{n+2}}p_{N_{n'}+2}(u_j).
\end{align*}

\item If $j\in [a_{n+1},a_{n+1}+\Delta_{n+1}-c_{n'}^k)$ then $T^{c_{n'}^{k'}}u_j=u_{j+c_{n'}^{k'}}$ and \eqref{but} follows from \eqref{tail1} and \eqref{tail2}.

\item If $j\in [a_{n+1}-c_{n'}^{k'}, a_{n+1})$ then
\[T^{c_{n'}^{k'}}u_j=T^{j+c_{n'}^{k'}}u_{0}=u_{j+c_{n'}^{k'}}+T^{j+c_{n'}^{k'}-a_{n+1}}u_0.\]
since $a_{n+1}-c_{n'}^{k'}>c_n^{\mu_n}+t_n$ and $j+c_{n'}^{k'}\in [a_{n+1},a_{n+1}+c_{n'}^{k'})\subset [a_{n+1},a_{n+1}+\Delta_{n+1})$. It follows from \eqref{tail1}, \eqref{tail2} and \eqref{tail4}  that
\begin{align*}
 p_{N_{n'}}(T^{c_{n'}^{k'}}u_j)&\le \frac{1}{2^{j+2}C_{n+2}}p_{N_{n'}+2}(u_j)+\frac{1}{2^{a_{n+1}+1}C_{n+2}}p_{1}(u_j)\\
&\le \frac{1}{2^{j+1}C_{n+2}}p_{N_{n'}+2}(u_j).
\end{align*}

\item If $j\in [c_n^{\mu_n}+t_n,a_{n+1}-c_{n'}^{k'})$ then 
$T^{c_{n'}^{k'}}u_j=u_{j+c_{n'}^{k'}}$ and \eqref{but} follows from \eqref{tail1} and \eqref{tail2}.

\item  If $j\in [c_n^{k},c_n^{k}+t_n)$ for some $1\le k\le \mu_n$ then we have to investigate different cases: 
\begin{itemize}
\item If $n'=n$ then
\[T^{c_{n'}^{k'}}u_j=u_{j+c_{n'}^{k'}}-T^{c_{n'}^{k'}}T^{j-c_n^{k}}Q_{n,k}(T)u_0\]
because if we let $K=\max\{k,k'\}$ then \[j+c_{n'}^{k'}\in [c_n^{k}+c_{n}^{k'},c_n^{k}+c_{n}^{k'}+t_n)\subset [c_n^{K}+t_n,2c_n^K+t_n)\] and $2c_n^{K}+t_n<c_{n}^{K+1}$ if $K<\mu_n$ and $2c_n^{\mu_n}+t_n<a_{n+1}$. Moreover, since
$c_{n}^{k'}\le \text{val}(T^{c_{n'}^{k'}}T^{j-c_n^{k}}Q_{n,k}(T))\le \text{deg}(T^{c_{n'}^{k'}}T^{j-c_n^{k}}Q_{n,k}(T))\le c_{n}^{k'}+2t_n+n$, it follows from \eqref{tail2} and \eqref{tail6} that
\begin{align*}
 p_{N_{n'}}(T^{c_{n'}^{k'}}u_j)&\le \frac{1}{2^{j+2}C_{n+2}}p_{N_{n'}+2}(u_j)+D_n \sup_{m\in [c_{n}^{k'},c_{n}^{k'}+2t_n+n]} p_{N_{n'}}(T^m u_0)\\
 &\le \frac{1}{2^{j+2}C_{n+2}}p_{N_{n'}+2}(u_j)+\frac{1}{2^{c_{n}^{k}+t_n+2}C_{n+2}}p_{N_{n'}+2}(u_j)\\
&\le \frac{1}{2^{j+1}C_{n+2}}p_{N_{n'}+2}(u_j).
\end{align*}
\item If $n'<n$ and $j\in [c_n^{k}+t_n-c_{n'}^{k'},c_n^{k}+t_n)$ then 
\[T^{c_{n'}^{k'}}u_j=u_{j+c_{n'}^{k'}}-T^{c_{n'}^{k'}}T^{j-c_n^{k}}Q_{n,k}(T)u_0\]
since  $j+c_{n'}^{k'}\in [c_{n}^{k}+t_n,c_n^{k}+2t_n)$ and $c_n^{k}+2t_n<c_{n}^{k+1}$ if $k<\mu_n$ and $c_n^{\mu_n}+2t_n<a_{n+1}$. Moreover, since $t_n\le\text{val}(T^{c_{n'}^{k'}}T^{j-c_n^{k}})\le \text{deg}(T^{c_{n'}^{k'}}T^{j-c_n^{k}})< 2t_n$, it follows from \eqref{tail1} and \eqref{tail5} that
\begin{align*}
 p_{N_{n'}}(T^{c_{n'}^{k'}}u_j)&\le \frac{1}{2^{j+2}C_{n+2}}p_{N_{n'}+2}(u_j)+D_n \sup_{m\in [t_n,3t_n+n)} p_{n}(T^m u_0)\\
 &\le \frac{1}{2^{j+2}C_{n+2}}p_{N_{n'}+2}(u_j)+\frac{1}{2^{c_n^{k}+t_n+2}C_{n+2}}p_{1}(u_j)\\
&\le \frac{1}{2^{j+1}C_{n+2}}p_{N_{n'}+2}(u_j).
\end{align*}
\item If $n'<n$ and $j\in [c_n^{k},c_n^{k}+t_n-c_{n'}^{k'})$ then 
$T^{c_{n'}^{k'}}u_j=u_{j+c_{n'}^{k'}}$
 and \eqref{but} follows from \eqref{tail1}. 
\end{itemize}

\item Finally, if $j\in [c_n^{k-1}+t_n,c_n^{k})$ for some $1\le k\le \mu_n$ or if $j=t_n$,  we investigate several cases:
\begin{itemize}
\item If $n'=n$ and $k'\ge k$ then $T^{c_{n'}^{k'}}u_j=u_{j+c_{n'}^{k'}}$ since \[c_{n'}^{k'}+j\in [c_{n}^{k'}+t_n, c_{n}^{k'}+c_{n}^{k})\subset [c_{n}^{k'}+t_n, c_{n}^{k'+1})\]
and \eqref{but} follows from \eqref{tail2}.  
\item If $n'= n$, $k'<k$ and $j\in [c_n^{k}-c_{n'}^{k'}+t_n,c_n^{k})$, we also get
$T^{c_{n'}^{k'}}u_j=u_{j+c_{n'}^{k'}}$
since $c_n^{k}-c_{n'}^{k'}+t_n\ge c_n^{k-1}+t_n$ and since $j+c_{n'}^{k'}\in [c_n^k+t_n,c_n^{k+1})$ if $k<\mu_n$ and $j+c_{n'}^{k'}\in [c_n^{\mu_n}+t_n,a_{n+1})$ otherwise. Condition \eqref{but} then follows from \eqref{tail2}.

\item If $n'=n$, $k'<k$ and $j\in [c_n^{k}-c_{n'}^{k'},c_n^{k}-c_{n'}^{k'}+t_n)$, we get
\[T^{c_{n'}^{k'}}u_j=u_{j+c_{n'}^{k'}}+T^{c_{n'}^{k'}}T^{j-c_n^{k}}Q_{n,k}(T)u_0\]
since $c_n^{k}-c_{n'}^{k'}>c_n^{k-1}+t_n$. Moreover, since $\text{deg}( T^{c_{n'}^{k'}}T^{j-c_n^{k}}Q_{n,k}(T))\le 2t_n+n$, it follows from \eqref{tail2} and \eqref{tail7} that
\begin{align*}
 p_{N_{n'}}(T^{c_{n'}^{k'}}u_j)&\le \frac{1}{2^{j+2}C_{n+2}}p_{N_{n'}+2}(u_j)+D_n \sup_{m\le 2t_n+n} p_{n}(T^m u_0)\\
 &\le \frac{1}{2^{j+2}C_{n+2}}p_{N_{n'}+2}(u_j)+\frac{1}{2^{c_n^{k}+1}C_{n+2}}p_{1}(u_j)\\
&\le \frac{1}{2^{j+1}C_{n+2}}p_{N_{n'}+2}(u_j).
\end{align*}

\item In a similar way, if $n'<n$  and if $j\in [c_n^{k}-c_{n'}^{k'},c_n^{k})$ then
\[T^{c_{n'}^{k'}}u_j=u_{j+c_{n'}^{k'}}+T^{c_{n'}^{k'}}T^{j-c_n^{k}}Q_{n,k}(T)u_0\]
and since $\text{deg}( T^{c_{n'}^{k'}}T^{j-c_n^{k}}Q_{n,k'}(T))\le 2t_n+n$, \eqref{but} follows from \eqref{tail1} and \eqref{tail7}.
\item Finally, if $(n',k')<(n,k)$ and if $j\in [c_n^{k-1}+t_n,c_n^{k}-c_{n'}^{k'})$ then 
$T^{c_{n'}^{k'}}u_j=u_{j+c_{n'}^{k'}}$ and \eqref{but} follows from \eqref{tail1}.
\end{itemize}
\end{itemize}
\end{proof}

Thanks to all these results, we are now able to write a list of conditions such that if the sequence $(u_n)_{n\ge 0}$ satisfies each of these conditions then $T$ admits no non-trivial invariant subset.

\begin{lemma}[Final result]\label{finalresult}
Under the assumptions of Lemmas~\ref{cont}, \ref{Kn} and \ref{tail}, if for every $n\ge 1$, we have
\begin{enumerate}
\item for every $0\le j< t_n$,
\begin{equation}
\label{final1}
p_{N_n}(u_{c_n^k+j})\le \frac{1}{D_n}
\end{equation}
\item 
\begin{equation}
\label{final2}
p_{N_n}(u_{a_n})\to 0
\end{equation}
\item for every  $t_n\le j\le 2t_n$
\begin{equation}
\label{final3}
p_{N_n}(u_j)\le \frac{1}{2^nD_n}
\end{equation}
\end{enumerate}
 then every non-zero vector in $X$ is hypercyclic for $T$ and thus $T$ has no non-trivial invariant subset.
\end{lemma}
\begin{proof}
Let $x\in X$ such that $p_1(x)=1$. Let $z\in X$ and $N\ge 1$. Since we have assumed that $p_{n+1}\ge 2 p_n$ for every $n\ge 1$, it suffices to show that there exists $i\ge 0$ such that $p_N(T^ix-z)\le 10$ in order to deduce that every non-zero vector is hypercyclic.

Since $u_0$ is cyclic, there exists a polynomial $S$ such that
\[p_N(S(T)u_0-z)\le 1.\]
We can then find $n$ as large as desired such that 
\begin{itemize}
\item $R_{n}=\max\{\text{deg}(S),\lceil|S|\rceil\}$,
\item $N_n=N+2R_n$,
\item $y:=\pi_{[0,t_n)}x\in K_n$,
\item $p_{N_n+3}(x-y)=p_{N_n+3}(\pi_{[t_n,+\infty)}x)\le 1$.
\item $p_{N_n}(u_{a_n})\le \frac{1}{2^n}$.
\end{itemize}

By the definition of net $\mathcal{S}_n$, there exists $S_{n,w}\in \mathcal{S}_n$ such that
$|S-S_{n,w}|\le \frac{1}{\max_{l\le R_n}p_{N_n}(T^lu_0)}$ with $\text{deg}S_{n,w}\le R_n$ and $|S_{n,w}|\le R_n$. In particular, we have
\[p_{N_n}(S(T)u_0-S_{n,w}(T)u_0)\le |S-S_{n,w}| \max_{l\le R_n}p_{N_n}(T^lu_0)\le 1.\]

By Corollary~\ref{cor Pn}, there exists a polynomial $P_{n,l}$ such that 
\[p_{N_n}(P_{n,l}(T)y-u_0)\le 2p_{N_n}(u_{a_n})+D_n\max_{t_n\le j\le 2t_n}p_{N_n}(u_j)\le\frac{3}{2^n}\quad \text{by \eqref{final3}}\] and thus by Corollary~\ref{contnorm}
\begin{align*}
p_N\big(S_{n,w}(T)(P_{n,l}(T)y-u_0)\big)&\le R_n \max_{k\le R_n}p_{N}\big(T^k(P_{n,l}(T)y-u_0)\big)\\
&\le R_n \Big(\prod_{l=0}^{R_n-1}\max\{4C_{N+2l+1}L_{N+2l},1\}\Big)p_{N+2R_n}(P_{n,l}(T)y-u_0)\\
&\le R_n \Big(\prod_{l=0}^{R_n-1}\max\{4C_{N+2l+1}L_{N+2l},1\}\Big)p_{N_n}(P_{n,l}(T)y-u_0)\to 0
\end{align*}
as $n\to \infty$ since $R_n$ does not depend on $n$.\\

Let $k$ such that $Q_{n,k}=S_{n,w}P_{n,l}$, it follows from Lemma~\ref{tail} that
\[p_{N_n}(T^{c_n^k}(x-y))\le p_{N_n+3}(x-y)\le 1.\]
Since 
\begin{align*}
p_{N_n}(T^{c_n^k}x-z)&\le p_{N_n}\big(T^{c_n^k}(x-y)\big)+p_{N_n}(T^{c_n^k}y-Q_{n,k}(T)y)+p_{N_n}\big(S_{n,w}(T)(P_{n,l}(T)y-u_0)\big)\\
&\quad+ p_{N_n}(S_{n,w}(T)u_0-S(T)u_0)+ p_{N_n}(S(T)u_0-z),
\end{align*}
it remains to look at $p_{N_n}(T^{c_n^k}y-Q_{n,k}(T)y)$. If we write $y=\sum_{j=0}^{t_n-1}y_jT^ju_0$ then by Corollary~\ref{cor Pn}, we have 
\[\sum_{j=0}^{t_n-1}|y_j|\le D_n.\]
 Therefore, since
\[T^{c_n^k}y-Q_{n,k}(T)y=\sum_{j=0}^{t_n-1}y_j(T^{c_n^k+j}u_0-Q_{n,k}(T)T^ju_0)
=\sum_{j=0}^{t_n-1}y_j u_{c_n^k+j},\]
we deduce from \eqref{final1} that
\[p_{N_n}(T^{c_n^k}y-Q_{n,k}(T)y)\le D_{n}\max_{j< t_n}p_{N_n}(u_{c_n^k+j})\le 1.\]
We conclude that if $n$ is sufficiently big then $p_{N}(T^{c_n^k}x-z)\le 10$. In other words, $x$ is hypercyclic and therefore every non-zero vector is hypercyclic.
\end{proof}

It remains to prove that if $T$ satisfies the assumptions of Theorem~\ref{technique} then it is possible to construct a sequence $(u_j)_{j\ge 1}$ satisfying \eqref{tn}, \eqref{finalcont1}-\eqref{finalcont5}, \eqref{K1}, \eqref{tail1}-\eqref{tail7} and \eqref{final1}-\eqref{final3}.

\begin{proof}[End of the proof of Theorem~\ref{technique}]
We recall that if for every $j \le J$, every $m\le M$,
\begin{enumerate}
\item $p_{J-1}(\alpha_{1,m}e_{n_{1,m}})\le \varepsilon$ 
\item $p_1(\alpha_{J,m}e_{n_{J,m}})\ge \frac{1}{\varepsilon}$
\item for every $1\le l\le M$, if $\next^l(j,m)$ exists then for every $k< K$,
\[C p_k(\alpha_{\next^{l}(j,m)}e_{n_{\next^{l}(j,m)}})\le p_{k+1}(\alpha_{j,m}e_{n_{j,m}}),\]
\end{enumerate}
then it follows from (2) and (3) that $p_{J}(\alpha_{j,m}e_{n_{j,m}})\ge \frac{1}{\varepsilon}$ for every $j\le J$ and every $m\le M$. Condition (3) will be mainly used to satisfy the condition of continuity \eqref{finalcont1} and the conditions \eqref{tail1} and \eqref{tail2}.

Assume that $u_0,...u_{t_n-1}$ have been chosen and that $t_n$ is a multiple of $n+1$. We let $c_n=t_n+(n+1)(2t_n+n)$, $a_{n+1}=(n+1)\Delta_{n+1}+c_n^{\mu_n}+t_n$ and we recall that $\Delta_{n+1}=c_{n}^{\mu_n}+t_n$.  Each inequality concerning these parameters is then satisfied (see \eqref{param}).

We start by choosing elements $u_j$ for $j\in \bigcup_{k=1}^{\mu_n}[c_n^k,c_n^{k}+t_n)$. For theses indices, we mainly want that $p_{N_n}(u_j)$ is small 
and that $p_{N_{n}+2}(u_j)$ is big. 
The elements $u_j$ for $j\in [c_n^k,c_n^{k}+t_n)$ are chosen such that
\begin{enumerate}
\item $p_{l+1}(u_j)\ge 2^{j+1} p_l(u_{j+1})$ for every $l\le n$, every $j\in [c_n^{k},c_n^{k}+t_n-1)$ (cf. \eqref{finalcont1})
\item $p_{N_{n+1}+2}(u_j)\ge nD_n2^j\sup_{l< t_n}p_{1}(T^lu_0)$ for every $j\in [c_n^{k},c_n^{k}+t_n)$ (cf. \eqref{K1})
\item $p_{l+2}(u_j)\ge 2^{j+2}C_{n+2}p_{l}(u_{j+r})$ for every $l\le n$, every $1\le r\le c_{n-1}^{\mu_{n-1}}$, every $j\in [c_n^{k},c_n^{k}+t_n-r)$ (cf. \eqref{tail1}).
\item $p_{N_{n}+2}(u_j)\ge  2^{c_n^{k}+t_n+2}C_{n+2}$ for every $j\in [c_n^{k},c_n^{k}+t_n)$ (cf. \eqref{tail6})
\item $p_{N_n}(u_j)\le \frac{1}{D_n}$ for every $j\in [c_n^{k},c_n^{k}+t_n)$ (cf. \eqref{tail6} and \eqref{final1})
\end{enumerate}
To this end, we use the assumptions of Theorem~\ref{technique} with 
$\varepsilon$ sufficiently small, $C=2^{c_{n}^k+t_n+1}C_{n+2}$, $J=N_n+1$, $K=n+1$, $M=t_n$
 and $N$ which is the maximal index of elements of the basis $(e_k)$ which have already been used.  We can then let $u_{c_n^{k}+m-1}=\alpha_{1,m}e_{n_{1,m}}$ and since $M=t_n$ and $N_{n+1}+2\ge N_n+1$, we get the desired inequalities.\\

We now choose elements $u_j$ for $j\in [t_n,c_n)$. We mainly need that $p_{n}(u_j)$ is small for $j\in [t_n,3t_n+n)$ 
and that $p_{1}(u_j)$ is big for $j\in [c_n-c_{n-1}^{\mu_{n-1}},c_n)$.
The elements $u_{t_n},\dots,u_{c_n-1}$ will be chosen such that
\begin{enumerate}
\item $p_{l+1}(u_j)\ge 2^{j+1} p_l(u_{j+1})$ for every $l\le n$, every $j\in [t_n,c_n-1)$ (cf. \eqref{finalcont1})
\item $p_{1}(u_{c_n-1})\ge \max\{2^{c_n}p_{n}(u_{c_n}),D_n 2^{c_n}\sup_{j<t_n}p_n(T^ju_0)\}$ (cf. \eqref{finalcont1} and \eqref{finalcont4})
\item $p_{n}(u_{t_n})\le \frac{1}{2^{t_n}}p_{1}(u_{t_n-1})$ (cf. \eqref{finalcont1})
\item $p_{n}(u_{j})\le \frac{1}{2^{c_n^{\mu_n}+t_n}D_n}\min\{p_1(u_{l}):l\in\bigcup_{k=1}^{\mu_n}\{c_n^{k}+t_n-1\}\}$ for every $j\in [t_n,2t_n+n]$ (cf. \eqref{finalcont5}) 
\item $p_{l+2}(u_j)\ge 2^{j+1}p_{l}(u_{j+r})$ for every $l\le n$, every $1\le r\le c_{n-1}^{\mu_{n-1}}$, every $j\in [t_n,c_n-r)$ (cf
. \eqref{tail1})
\item $p_1(u_j)\ge 2^{c_n+1}\max_{l\in[c_n,c_n+t_n)}p_n(u_l)$ for every $1\le r\le c_{n-1}^{\mu_{n-1}}$, every $j\in [c_n-c_{n-1}^{\mu_{n-1}},c_n)$ (cf. \eqref{tail1})
\item $p_{n}(u_j)\le 2^{t_n+1}\min_{t_n-c_{n-1}^{\mu_{n-1}}\le l< t_n} p_1(u_l)$ for every $j\in [t_n,t_n+c_{n-1}^{\mu_{n-1}}]$ (cf. \eqref{tail1} and \eqref{tail2})
\item $p_{n}(u_{j})\le \frac{1}{2^{c_n^{\mu^n}+t_n+2}C_{n+2}D_n}\min\{p_1(u_j):j\in \bigcup_{k=1}^{\mu_n}[c_n^k,c_n^k+t_n)\}$ for every $j\in [t_n,3t_{n}+n)$ (cf. \eqref{tail5})
\item $p_1(u_j)\ge 2^{c_n+2}C_{n+2}D_n\sup_{m<t_n}p_n(T^mu_0)$ for every $j\in [c_n-c_{n-1}^{\mu_{n-1}},c_{n})$ (cf. \eqref{tail7}).
\item $p_n(u_j)\le \sup_{m<t_n}p_n(T^mu_0)$ for every $j\in [t_n,2t_n+n)$ (cf. \eqref{tail7}) 
\item $p_{N_n}(u_j)\le \frac{1}{2^nD_n}$ for every $j\in [t_n,2t_n]$ (cf. \eqref{final3}).
\end{enumerate}
To this end, we use the assumptions of Theorem~\ref{technique} with $\varepsilon$ sufficiently small, $J=n+1$, $K=n+1$, $M=(c_n-t_n)/(n+1)$ and $N$ which is the maximal index  of elements of the basis $(e_k)$ which have already been used. We can then let $u_{t_n+(j-1)M+m-1}=\alpha_{j,m}e_{n_{j,m}}$ and since $M\ge 2t_n+n$, we get the desired inequalities. We remark that it follows from (11) that $\sup_{m\le 2t_n+n}p_n(T^mu_0)\le \sup_{m<t_n}p_n(T^mu_0)$ and thus that \eqref{tail7} is satisfied for $k=1$.\\

Let $1\le k< \mu^n$. We complete by choosing the elements $u_j$ for $j\in [c_n^k+t_n,c_n^{k+1})$. We want that $p_{n}(u_j)$ is small for $j\in [c^k_n+t_n,2c^k_n+t_n)$ and that $p_{1}(u_j)$ is big for $j\in [c^{k+1}_n-c_{n}^{k},c^{k+1}_n)$. 
The elements $u_{c_n^{k}+t_n},\dots,u_{c_n^{k+1}-1}$ are chosen such that
\begin{enumerate}
\item $p_{l+1}(u_j)\ge 2^{j+1} p_l(u_{j+1})$ for every $l\le n$, every $j\in [c_n^{k}+t_n,c_n^{k+1}-1)$ (cf. \eqref{finalcont1})
\item $p_{1}(u_{c_n^{k+1}-1})\ge \max\{2^{c_n^{k+1}}p_{n}(u_{c_n^{k+1}}),D_n 2^{c_n^{k+1}}\sup_{l<t_n}p_n(T^lu_0)\}$ (cf. \eqref{finalcont1} and \eqref{finalcont4})
\item $p_{n}(u_{c^k_n+t_n})\le \frac{1}{2^{c^k_n+t_n}}p_{1}(u_{c^k_n+t_n-1})$ (cf. \eqref{finalcont1})
\item $p_1(u_{c_{n}^{k+1}-1})\ge 2^{c_n^{\mu_n}+t_n}D_n \sup_{j\in [t_n,2t_n+n]}p_{n}(u_{j})$ (cf. \eqref{finalcont5})
\item $p_{l+2}(u_j)\ge 2^{j+2}p_{l}(u_{j+r})$ for every $l\le n$, every $r\le c_{n}^{k}$, every $j\in [c_n^{k}+t_n,c_n^{k+1}-r)$ (cf. \eqref{tail1} and \eqref{tail2})
\item $p_{1}(u_j)\ge 2^{c_n^{k+1}+1} \max_{l\in [c_n^{k+1},c_n^{k+1}+t_n)}p_{n}(u_{l})$ for every $j\in [c_n^{k+1}-c_n^k,c_n^{k+1})$ (cf. \eqref{tail1} and \eqref{tail2})
\item $p_{n}(u_j)\le \frac{1}{2^{c_n^k+t_n+1}} \min_{l\in [0,c_n^k+t_n)}p_{1}(u_{l})$ for every $j\in [c_n^{k}+t_n,2c_n^{k}+t_n)$ (cf. \eqref{tail1} and \eqref{tail2})
\item $p_{N_n}(u_j)\le \frac{1}{D_n}$ for every $j\in [c_n^{k}+t_n,c_n^{k}+2t_n+n]$ (cf. \eqref{tail6})
\item $p_1(u_j)\ge 2^{c_n^{\mu_n}+1}C_{n+2}D_n\sup_{m\le 2t_n+n}p_n(T^mu_0)$ for every $j\in [c_n^{k+1}-c_n^{k},c_n^{k+1})$ (cf. \eqref{tail7})
\end{enumerate}
To this end, we use the assumptions of Theorem~\ref{technique} with 
$\varepsilon$ sufficiently small, $C=2^{c_{n}^{k+1}+1}C_{n+2}$, $J=n+1$, $K=n+1$, $M=\frac{c_n^{k+1}-c_n^k-t_n}{n+1}$
 and $N$ which is the maximal index of elements of the basis $(e_k)$ which have already been used.  We can then let $u_{c_n^{k}+t_n+(j-1)M+m-1}=\alpha_{j,m}e_{n_{j,m}}$ and since $M\ge c_n^{k}$, we get the desired inequalities.\\

We are now looking for the elements $u_j$ for $j\in [a_{n+1},a_{n+1}+\Delta_{n+1})$. For these indices, we want that
$p_{N_{n}+2}$ is big and that $p_{N_{n+1}}(a_{n+1})$ is small. This is possible since $N_{n+1}-N_n\le 1$ and thus $N_{n}+2>N_{n+1}$. The elements $u_{a_{n+1}},\dots,u_{a_{n+1}+\Delta_{n+1}-1}$ are chosen such that 
\begin{enumerate}
\item $p_{l+1}(u_j)\ge 2^{j+1} p_l(u_{j+1})$ for every $l\le n$, every $j\in [a_{n+1},a_{n+1}+\Delta_{n+1}-1)$ (cf. \eqref{finalcont1})
\item $p_{N_{n+1}+2}(u_j)\ge nD_n2^{a_{n+1}+\Delta_{n+1}}\sup_{l<t_n}p_{1}(T^lu_0)$ for every $j\in [a_{n+1},a_{n+1}+\Delta_{n+1})$(cf. \eqref{K1})
\item $p_{l+2}(u_j)\ge 2^{j+1}p_{l}(u_{j+r})$ for every $l\le n$, every $1\le r\le c_n^{\mu_n}$ and $j\in [a_{n+1}, a_{n+1}+\Delta_{n+1}-r)$ (cf. \eqref{tail1} and \eqref{tail2})
\item $p_{N_{n+1}}(u_{a_{n+1}})<\frac{1}{n+1}$ (cf. \eqref{final2})
\end{enumerate}
We use the assumptions of Theorem~\ref{technique} with $\varepsilon$ sufficiently small, $C=2^{a_{n+1}+\Delta_!{n+1}+1}C_{n+2}$, $J=N_n+2$, $K=n+2$, $M=\Delta_n$ and $N$ which is the maximal index of elements of the basis $(e_k)$ which have already been used. We can then let $u_{a_{n+1}+m-1}=\alpha_{1,m}e_{n_{1,m}}$ and since $M\ge c_n^{\mu_n}$, we get the desired inequalities.\\

We complete by choosing the elements $u_j$ for $j\in [c_n^{\mu_n}+t_n,a_{n+1})$. We want that $p_{n}(u_j)$ is small for $j\in [c_n^{\mu_n}+t_n,\Delta_{n+1}+c_n^{\mu_n})$ and that $p_{1}(u_j)$ is big for $j\in [a_{n+1}-c_n^{\mu_n} ,a_{n+1})$. 
The elements $u_{c_n^{\mu_n}+t_n},\dots,u_{a_{n+1}-1}$ are chosen such that
\begin{enumerate}
\item $p_{l+1}(u_j)\ge 2^{j+1} p_l(u_{j+1})$ for every $l\le n$, every $j\in [c_n^{\mu^n}+t_n,a_{n+1}-1)$ (cf. \eqref{finalcont1})
\item $p_{1}(u_{a_{n+1}-1})\ge 2^{a_{n+1}}\max\{p_n(u_{a_{n+1}}),p_n(u_0)\}$ (cf. \eqref{finalcont1} and \eqref{finalcont2})
\item $p_{n}(u_{\Delta_{n+1}})\le \min\{\frac{p_1(u_{c_n^{\mu_n}+t_n-1})}{2^{\Delta_{n+1}}} ,\frac{p_1(u_{a_{n+1}+\Delta_{n+1}-1})}{2^{a_{n+1}+\Delta_{n+1}}}\}$ (cf. \eqref{finalcont1} and \eqref{finalcont3})
\item $p_{l+2}(u_j)\ge 2^{j+2}C_{n+2}p_{l}(u_{j+r})$ for every $l\le n$, every $1\le r\le c_n^{\mu_n}$, every $j\in [c_n^{\mu^n}+t_n,a_{n+1}-r)$ (cf. \eqref{tail1} and \eqref{tail2})
\item $p_1(u_j)\ge 2^{a_{n+1}+1}C_{n+2}\max\{p_{n}(u_l):l\in [a_{n+1},a_{n+1}+c_n^{\mu_n})\}$ for every $j\in [a_{n+1}-c_n^{\mu_n},a_{n+1})$ (cf. \eqref{tail1} and \eqref{tail2})
\item $p_{n}(u_j)\le \frac{\min\{p_1(u_l):l\in [t_n,c_n^{\mu_n}+t_n)\}}{2^{\Delta_{n+1}+1}C_{n+2}}$ for every $j\in [c_n^{\mu_n}+t_n,\Delta_{n+1}+c_{n}^{\mu_n}]$ (cf. \eqref{tail1} and \eqref{tail2})
\item $p_n(u_j)\le \frac{\min\{p_1(u_l): l\in[a_{n+1}+\Delta_{n+1}-c_n^{\mu_n},a_{n+1}+\Delta_{n+1})\}}{2^{a_{n+1}+\Delta_{n+1}+1}C_{n+2}}$ for every $j\in [\Delta_{n+1},\Delta_{n+1}+c_n^{\mu_n})$ (cf. \eqref{tail3})
\item $p_1(u_j)\ge 2^{a_{n+1}+1}C_{n+2}\sup_{j\le c_n^{\mu_n}}p_n(T^j u_0)$ for every $j\in [a_{n+1}-c_n^{\mu_n},a_{n+1})$ (cf. \eqref{tail4})
\item $p_{N_n}(u_j)\le \frac{1}{D_n}$ for every $j\in [c_n^{\mu^n}+t_n,c_n^{\mu^n}+2t_n+n)$ (cf. \eqref{tail6})
\end{enumerate}
To this end, we apply conditions of Theorem~\ref{technique} with $\varepsilon$ sufficiently small, $C=2^{a_{n+1}+1}C_{n+2}$, $J=n+1$, $K=n+1$, $M=\Delta_{n+1}$ and $N$ which is the maximal index of elements of the basis $(e_k)$ which have already been used. We can then let $u_{c_n^{\mu_n}+t_n+(j-1)M+m-1}=\alpha_{j,m}e_{n_{j,m}}$ and since $M\ge c_n^{\mu_n}$, we get the desired inequalities. We remark that \eqref{tail6} is satisfied since
${\sup_{m\in \bigcup_{1\le k'\le \mu_n}[c_n^{k'},c_{n}^{k'}+2t_n+n]}p_{N_n}(u_m)\le \frac{1}{D_n}}$ and for every $j\in \bigcup_{k=1}^{\mu_n}[c_n^k,c_n^{k}+t_n)$, we have $p_{N_{n}+2}(u_j)\ge 2^{c_n^{k}+t_n+2}C_{n+2}$.\\

Finally, we choose $t_{n+1}>a_{n+1}+\Delta_{n+1}+c_n^{\mu_n}$ such that $t_{n+1}$ is a multiple of $n+2$ and such that we can complete the elements $u_0,\dots,u_{a_{n+1}+\Delta_{n+1}-1}$ so that \eqref{tn} is satisfied. We then let $u_{a_{n+1}+\Delta_{n+1}+j}=\alpha_j e_{m_j}$ for every $0\le j< t_{n+1}-a_{n+1}-\Delta_{n+1}$ where $(m_j)_{j}$ enumerates the missing elements of the family $(e_m)_{m< t_{n+1}}$ and where $(\alpha_j)_{j}$ are  non-zero scalars satisfying
 \begin{enumerate}
 \item $p_{l+1}(u_j)\ge 2^{j+1} p_l(u_{j+1})$ for every $l\le n$, every $j\in [a_{n+1}+\Delta_{n+1},t_{n+1}-1)$ (cf. \eqref{finalcont1})
\item $p_{n}(u_{a_{n+1}+\Delta_{n+1}})\le \frac{p_1(u_{a_{n+1}+\Delta_{n+1}-1})}{2^{a_{n+1}+\Delta_{n+1}}}$ (cf. \eqref{finalcont1})
\item $p_{l+2}(u_j)\ge 2^{j+2}C_{n+2}p_{l}(u_{j+r})$ for every $l\le n$, every $1\le r\le c_n^{\mu_n}$, every $j\in [a_{n+1}+\Delta_{n+1},t_{n+1}-r)$ (cf. \eqref{tail1} and \eqref{tail2})
\item $p_{n}(u_j)\le \frac{\min\{p_1(u_l):l\in [a_{n+1}+\Delta_{n+1}-c_n^{\mu_n},a_{n+1}+\Delta_{n+1})\}}{2^{a_{n+1}+\Delta_{n+1}+1}C_{n+2}}$ for every $j\in [a_{n+1}+\Delta_{n+1},a_{n+1}+\Delta_{n+1}+c_{n}^{\mu_n})$ (cf. \eqref{tail1} and \eqref{tail2})
 \end{enumerate}
These conditions are satisfied if $|\alpha_j|$ are small and decrease sufficiently rapidly.\\

It can now be verified that the sequence $(u_n)_{n\ge 0}$ has been constructed so that \eqref{tn}, \eqref{finalcont1}-\eqref{finalcont5}, \eqref{K1}, \eqref{tail1}-\eqref{tail7} and \eqref{final1}-\eqref{final3} are satisified and it thus follows from Lemma~\ref{finalresult} that $X$ does not satisfy the Invariant Subset Property.
\end{proof}

\section{Further remarks on Fréchet spaces without continuous norm}

Let $(X,(p_j)_j)$ be a separable infinite-dimensional Fréchet space with a Schauder basis $(e_n)_{n\ge 0}$. If $X$ does not possess a continuous norm, we have succeeded to show that if $\ker p_{j+1}$ is a subspace of finite codimension in $\ker p_j$ then $X$ satisfies the Invariant Subspace Property (Theorem~\ref{omega}) and that if for every $j\ge 1$, $\ker p_{j+1}$ is a subspace of infinite codimension in $\ker p_j$ then $X$ does not satisfy the Hereditary Invariant Subset Property (Corollary~\ref{heredker}). However, we do not know  if in the second case, $X$ can satisfy the Invariant Subset Property. Indeed, none of our approaches seems to be working for Fréchet spaces where $\ker p_{j+1}$ is a subspace of infinite codimension in $\ker p_j$ . 

In the approach of Section~\ref{Sec1}, we have worked with a special family $(x^{(l)})_l$ and we have modified this family so that each series $\sum_{l=1}^{\infty}\beta_l x^{(l)}$ exists for every sequence $(\beta_l)$. This could be easily done in the proof of Theorem~\ref{omega} when $\ker p_{j+1}$ is a subspace of finite codimension in $\ker p_j$ since it sufficed to cancel a finite number of coordinates. However, if $\ker p_{j+1}$ is a subspace of infinite codimension in $\ker p_j$, it could be necessary to cancel an infinite number of coordinates and this seems difficult to guarantee. 

On the other hand, if for every $j\ge 1$, $\ker p_{j+1}$ is a subspace of infinite codimension in $\ker p_j$ then we can find a sequence of seminorms $(q_j)$ inducing the same topology than $(p_j)$ such that for every $\varepsilon>0$, every $C,M,N\ge 1$, every $K\ge J\ge 2$, there exists $(n_{j,m})_{j\le J, m\le M}\subset]N,\infty[$ and a sequence $(\alpha_{j,m})_{j\le J, m\le M}$ of non-zero scalars
such that for every $j \le J$, every $m\le M$,
\begin{enumerate}
\item if $j<J$ then $p_{J-j}(\alpha_{j,m}e_{n_{j,m}})=0$,
\item $p_1(\alpha_{J,m}e_{n_{J,m}})\ge \frac{1}{\varepsilon}$,
\item for every $1\le l\le M$, if $\next^l(j,m)$ exists then for every $k< K$,
\[ C p_k(\alpha_{\next^{l}(j,m)}e_{n_{\next^{l}(j,m)}})\le p_{k+1}(\alpha_{j,m}e_{n_{j,m}}).\]
\end{enumerate}
We can thus hope to be able to construct a Read-type operator whose each non-zero vector is hypercyclic as this is done in the proof of Theorem~\ref{technique}. However, an important part of the proof relies on the fact that the sets $K_n$ are compact where
\begin{equation*}
K_n=\Big\{y\in \text{span}(u_0,\dots,u_{t_{n}-1})~:~ p_{1}(y)\le \frac{3}{2}\ \text{and}\  p_{1}(\tau_n y)\ge 1/2\Big\}
\end{equation*}
and unfortunately, if $p_1$ is not a norm, it cannot be guarantee that these sets $K_n$ are compact.
We thus ask the following question:

\begin{problem}
If $(X,(p_j)_j)$ is a separable infinite-dimensional Fréchet space with a Schauder basis $(e_n)$ such that  for every $j\ge 1$, $\ker p_{j+1}$ is a subspace of infinite codimension in $\ker p_j$, does this space satisfy the Invariant Subspace/Subset Property?
\end{problem}

\section*{Acknowledgments}
The author is grateful to Aharon Atzmon, José Bonet, Sophie Grivaux and Alfred Peris for interesting comments on this paper.


\begin{thebibliography}{HD}

\normalsize
\baselineskip=17pt

\bibitem{Argyros1} Argyros, S. A.; Haydon, R. G. A hereditarily indecomposable $\mathscr L_\infty$-space that solves the scalar-plus-compact problem. Acta Math. 206 (2011), no. 1, 1--54.
\bibitem{Argyros2} Argyros, S. A.; Motakis, P. A reflexive hereditarily indecomposable space with the hereditary invariant subspace property. Proc. Lond. Math. Soc. (3) 108 (2014), no. 6, 1381--1416.
\bibitem{Atzmon} Atzmon, A. An operator without invariant subspaces on a nuclear Fréchet space. Ann. of Math. (2) 117 (1983), no. 3, 669--694.
\bibitem{Atzmon2} Atzmon, A. Nuclear Fréchet spaces of entire functions with transitive differentiation. J. Anal. Math. 60 (1993), 1--19.
\bibitem{Bernstein} Bernstein, A. R.; Robinson, A. Solution of an invariant subspace problem of K. T. Smith and P. R. Halmos. Pacific J. Math. 16 (1966) 421--431. 
\bibitem{Bonet} Pérez Carreras, P; Bonet, J. Barrelled locally convex spaces. North-Holland Mathematics Studies, 131, Notas de Matemática, Amsterdam, 1987. 
\bibitem{Diestel} Diestel, J. Sequences and series in Banach spaces. Graduate Texts in Mathematics, 92. Springer-Verlag, New York, 1984.
\bibitem{Enflo1} Enflo, P. On the invariant subspace problem in Banach spaces. Séminaire Maurey-Schwartz (1975-1976) Espaces $L\sp{p}$, applications radonifiantes et géométrie des espaces de Banach, Exp. Nos. 14-15, 7 pp. Centre Math., École Polytech., Palaiseau, 1976.
\bibitem{Enflo2} Enflo, P. On the invariant subspace problem for Banach spaces. Acta Math. 158 (1987), no. 3-4, 213--313.
\bibitem{Golinski1} Goli\'{n}ski, M. Invariant subspace problem for classical spaces of functions. J. Funct. Anal. 262 (2012), no. 3, 1251--1273.
\bibitem{Golinski2} Goli\'{n}ski, M. Operator on the space of rapidly decreasing functions with all non-zero vectors hypercyclic. Adv. Math. 244 (2013), 663--677.
\bibitem{Grivaux} Grivaux, S.; Roginskaya, M. A general approach to Read's type constructions of operators without non-trivial invariant closed subspaces. Proc. Lond. Math. Soc. 109 (2014), no. 3, 596--652.
\bibitem{Karl} Grosse-Erdmann, K-G.; Peris Manguillot, A. Linear chaos. Universitext. Springer, London, 2011.
\bibitem{Halmos} Halmos, P. R. Invariant subspaces of polynomially compact operators. Pacific J. Math. 16 (1966) 433--437.
\bibitem{Jarchow} Jarchow, H. Locally convex spaces. Mathematische Leitfäden, Stuttgart, 1981.
\bibitem{Johnson} Johnson, B. E.; Shields, A. L. Hyperinvariant subspaces for operators on the space of complex sequences. Michigan Math. J. 19 (1972), 189--191.
\bibitem{Korber} Körber, Karl-Heinz. Die invarianten Teil räume der stetigen Endomorphismen von $\omega $. (German) Math. Ann. 182 (1969) 95--103.
\bibitem{Lomonosov} Lomonosov, V. I. Invariant subspaces of the family of operators that commute with a completely continuous operator. (Russian) Funkcional. Anal. i Prilo\v{z}en. 7 (1973), no. 3, 55--56.
\bibitem{Meise} Meise R. and Vogt D. Introduction to functional analysis, Oxford University Press, New York, 1997.
\bibitem{Menet1} Menet, Quentin. Sous-espaces fermés de séries universelles sur un espace de Fréchet. (French) [Closed subspaces of universal series on a Fréchet space] Studia Math. 207 (2011), no. 2, 181--195
\bibitem{Read1} Read, C. J. A solution to the invariant subspace problem. Bull. London Math. Soc. 16 (1984), no. 4, 337--401.
\bibitem{Read2} Read, C. J. A solution to the invariant subspace problem on the space $l_1$. Bull. London Math. Soc. 17 (1985), no. 4, 305--317.
\bibitem{Read3} Read, C. J. The invariant subspace problem for a class of Banach spaces. II.
Hypercyclic operators. Israel J. Math. 63 (1988), 1--40.
\bibitem{Shields} Shields, Allen L. A note on invariant subspaces. Michigan Math. J. 17 (1970), 231--233.

\end{thebibliography}
\end{document}